\RequirePackage{fix-cm}
\documentclass[smallextended]{svjour3}       
\smartqed  
\usepackage{graphicx}
\usepackage{amsfonts}
\usepackage{amsmath}
      \def\JELname{{\bfseries JEL Classification}\enspace}
      \def\JEL#1{\par\addvspace\medskipamount{\rightskip=0pt plus1cm
      \def\and{\ifhmode\unskip\nobreak\fi\ $\cdot$
      }\noindent\JELname\ignorespaces#1\par}}

\spnewtheorem{ass}{Assumption}{\bf}{\it}

\makeatletter
\let\c@proposition\c@theorem
\let\c@remark\c@theorem
\let\c@lemma\c@theorem
\let\c@corollary\c@theorem
\let\c@definition\c@theorem
\let\c@example\c@theorem
\let\c@ass\c@theorem
\makeatother

 \numberwithin{equation}{section}
  \numberwithin{theorem}{section}
   \numberwithin{definition}{section}
    \numberwithin{example}{section}
    \numberwithin{proposition}{section}
       \numberwithin{lemma}{section}
    \numberwithin{remark}{section}
      \numberwithin{corollary}{section}
         \numberwithin{ass}{section}

\newcommand{\filt}[1]{\mathbb{#1}}
\newcommand{\sigalg}[1]{\mathcal{#1}}

\newcommand{\pare}[1]{\left(#1\right)}
\newcommand{\bra}[1]{\left[#1\right]}
\newcommand{\dbraco}[1]{[\kern-0.15em[ #1 [\kern-0.15em[}
\newcommand{\dbracc}[1]{[\kern-0.15em[ #1 ]\kern-0.15em]}
\newcommand{\set}[1]{\left\{#1\right\}}
\newcommand{\abs}[1]{\left\vert#1\right\vert}

\newcommand{\dfn}{:=}
\newcommand{\such}{\ | \ }

\newcommand{\sign}{\mathrm{sign}}
\newcommand{\ud}{\mathrm d}
\newcommand{\e}{\mathrm{e}}
\newcommand{\ess}{\mathrm{ess}}

\newcommand{\Pu}{{\mathbb{P}}}
\newcommand{\Qu}{{\mathbb Q}}
\newcommand{\Wu}{{\mathbb W}}
\newcommand{\E}{{\mathbb E}}
\newcommand{\Real}{\mathbb R}
\newcommand{\Natural}{\mathbb N}

\newcommand{\indic}{\mathbf{1}}

\newcommand{\B}{\mathcal{B}}
\newcommand{\Exp}{\mathcal E}
\newcommand{\X}{\mathcal{X}}
\newcommand{\R}{\mathcal{R}}

\newcommand{\nin}{n \in \Natural}
\newcommand{\kin}{k \in \Natural}

\begin{document}

\title{Filtration Shrinkage, the Structure of Deflators, and Failure of Market Completeness
}

\titlerunning{Filtration Shrinkage, Deflators, and Market Completeness}        

\author{Constantinos Kardaras         \and
        Johannes Ruf 
}


\institute{Constantinos Kardaras \at
	Department of Statistics\\
	London School of Economics and Political Science\\
	10 Houghton St\\
	London WC2A 2AE\\
              \email{k.kardaras@lse.ac.uk}           
           \and
          Johannes Ruf \at
	Department of Mathematics\\
	London School of Economics and Political Science\\
	10 Houghton St\\
	London WC2A 2AE\\
  \email{j.ruf@lse.ac.uk}           
}

\date{Received: date / Accepted: date}

\maketitle

\begin{abstract}
We analyse the structure of local martingale deflators projected on smaller filtrations. In a general continuous-path setting, we show that the local martingale part in the
multiplicative Doob-Meyer decomposition of projected local martingale deflators are themselves local martingale deflators in the smaller information market. Via use of a Bayesian filtering approach, we demonstrate the exact mechanism of how updates on the possible class of models under less information result in the strict supermartingale property of projections of such deflators. Finally, we demonstrate that these projections are unable to span all possible local martingale deflators in the smaller information market, by investigating a situation where market completeness is not retained under filtration shrinkage.
\JEL{C11 \and G13 \and G14}
\end{abstract}

\section{Introduction}

Optional projections of martingales to smaller filtrations retain the martingale property; for the class of local martingales, the latter preservation may fail. For instance, the projection of a nonnegative local martingale can only be guaranteed to be a supermartingale in the smaller filtration, but might fail to be a local martingale; see Stricker \cite{Stricker:1977} and F\"ollmer and Protter \cite{Foellmer_Protter_2010}.

Positive local martingales appear naturally as deflators in arbitrage theory. (See Sect.~\ref{S:notation} for definitions and a review of classical concepts in the theory of no-arbitrage.) Consider two nested, right-continuous filtrations $\filt{F} \subseteq\filt{G}$ and a continuous and $\filt{F}$-adapted process $S$, having the interpretation of the discounted price of a financial asset. Then, the existence of a strictly positive $\filt{G}$-local martingale $Y$ such that $Y S$ is also a $\filt{G}$-local martingale is equivalent to the so-called absence of arbitrage of the first kind.  If no such arbitrage opportunities are possible under $\filt{G}$, then the same is true under the smaller filtration $\filt{F}$; we refer to Sect.~\ref{S:notation} for a rigorous argument of this assertion. Hence, there must exist an $\filt{F}$-local martingale $L$   such that $L S$ is an $\filt{F}$-local martingale. Let now $\mathcal Y^{\filt{G}}$and $\mathcal Y^{\filt{F}}$  denote the set of all $\filt{G}$-adapted and $\filt{F}$-adapted local martingale deflators, respectively.  The above no-arbitrage considerations yield the implication 
\[
	\mathcal Y^{\filt{G}} \neq \emptyset \qquad \Longrightarrow \qquad \mathcal Y^{\filt{F}} \neq \emptyset.	
\]
It is natural to ask at this point if there is a direct way to construct an element of $Y^{\filt{F}}$ from a given $Y \in \mathcal Y^{\filt{G}}$.  The optional projection ${}^o Y$ of $Y$ on $\filt{F}$ is  not necessarily an $\filt{F}$-local martingale, as discussed above; hence it cannot be expected to be in $\mathcal Y^{\filt{F}}$. However, as our first main result, Theorem \ref{thm:1}, implies, the local martingale part $L$ of the \emph{multiplicative} Doob-Meyer $\filt{F}$-decomposition ${}^o Y = L(1-K)$ is an element of $\mathcal Y^{\filt{F}}$. Sect.~\ref{sec:proof1} contains the proof of Theorem~\ref{thm:1} and related results.

The previous motivates another natural question: when does the projection of $Y$ lose the local martingale property, that is, under which circumstances is it the case that $K_\infty > 0$?  In Sect.~\ref{sec:Bayes}, we investigate this question from a Bayesian viewpoint. As it turns out, whenever certain models (which were possible under the Bayesian prior) become impossible under the observed data (the stock price path, in this case), the projection of the deflator $Y$ loses the local martingale property, and $K$ increases.  In  Sect.~\ref{sec:dominating}, we generalise the Bayesian viewpoint, under the assumption that a certain dominating probability measure exists.

Markets admitting local martingale deflators are complete if, and only if, such a deflator is unique.  Since different local martingale deflators in $\mathcal Y^{\filt{G}}$ might have the same projection, it is easy to find an example such that a market is incomplete under $\filt{G}$ but complete under $\filt{F}$. Indeed, consider a complete market under $\filt{F}$ and add an independent Brownian motion to get to a filtration $\filt{G}$; then, the market is automatically incomplete under $\filt{G}$.

The reverse question is of more interest: given that the market is complete under $\filt{G}$, is it also complete under $\filt{F}$?  As it turns out, this is not always true; it is possible that certain $\filt{F}$-local martingale deflators do not result from the local martingale component of a projections of $\filt{G}$-local martingale deflators, and completeness in financial markets may be lost when we pass to smaller filtrations. We provide an explicit example in Sect.~\ref{sec:counterexample}.
This counterexample uses the L\'evy transformation $B$ of a standard Brownian motion $W$, namely
\[
B \dfn \int_0^\cdot \sign(W_u) \ud W_u = |W| - \Lambda,	
\]
where $\Lambda$ is the local time of $W$ at zero; see Revuz and Yor \cite[Theorem~VI.1.2]{RY}. In fact, we provide a rather general class of counterexamples, whose construction is of independent interest. To wit, let $\filt{F}^W$ and $\filt{F}^B$ denote the smallest right-continuous filtrations making $W$ and $B$ measurable, respectively. 
 Both $W$ and $B$ are standard Brownian motions with the predictable representation property on $\filt{F}^W$.  Furthermore, $B$ is a standard Brownian motion with the predictable representation property on $\filt{F}^B$, and it holds that $\filt{F}^B = \filt{F}^{|W|}$; see Jeanblanc et al.~\cite[Sect.~5.8.2]{JYC_2009}. The information lost when passing from $W$ to $B$ consists of the signs of the excursions of $W$; in view of Blumenthal \cite[ page~114]{Blumenthal:1992}, conditional on $\sigalg{F}^B_\infty = \sigalg{F}^{|W|}_\infty$, these signs are independent and identically distributed. Given that there are countably many excursions of $W$, there clearly exist non-deterministic $\sigalg{F}^W_\infty$-measurable random variables, which are independent of $\sigalg{F}^{B} _\infty$. As we argue in Theorem \ref{thm:main} in SubSect.~\ref{SS:5.2}, one may construct such random variables in an $\filt{F}^W$-adapted way: there exist $ \filt{F}^{W}$-stopping times with any prescribed probability law on the positive real line, independent of $\sigalg{F}^B_\infty$. The last result provides an interesting corollary: the existence of two nested filtrations $\filt{F} \subseteq\filt{G}$ and a one-dimensional continuous stock price process $S$, adapted to $\filt{F}$, such that the market is complete under $\filt{G}$ and under $\filt{F}$, but not under some ``intermediate information'' model. The material in SubSect.~\ref{SS:5.3} also yield a counterexample to a conjecture put forth in Jacod and Protter \cite{Jacod:Protter:2017}.

The complimentary problem of whether certain no-arbitrage and completeness conditions are preserved after filtration enlargement has been studied extensively is not considered in the present paper; we instead refer to \cite{Coculescu:2012}, \cite{Fontana:Jeanblanc:Song}, \cite{Jeanblanc:Song:2015}, \cite{Acciaio:Fontana:Kardaras}, \cite{Song:2016},    \cite{Aksamit:2017}, \cite{Song:2017}, \cite{Aksamit:2018}, \cite{Chau:Runggaldier:Tankov}, \cite{Fontana:2018}, \cite{Chau:Cosso:Fontana}, 
and the references therein. Filtration shrinkage and its effect on semimartingale characteristics has been studied by several researchers.  F\"ollmer and Protter \cite{Foellmer_Protter_2010}, Larsson \cite{Larsson_2013}, and Kardaras and Ruf \cite{Kardaras:Ruf:2018:filtration} consider reciprocals of Bessel
processes, projected on the filtration generated by one of its 
components, providing explicit examples where projections of nonnegative 
local martingale fail to be local martingales themselves.
Bielecki et al.~\cite{Bielecki:Jakubowski} discuss how the characteristics of semimartingales are related 
in different filtrations.  
Biagini et al.~\cite{Biagini:Mazzon} consider questions of 
``bubbles'' and arbitrage opportunities in absence of full information.

\section{Notation, Definitions, and Review of Classical Results} \label{S:notation}

In this section, we introduce the framework, and recall certain classical results which find use later on.

Fix a probability space $(\Omega, \sigalg{G}_\infty, \Pu)$, equipped with two right-continuous filtrations $\filt{F} \equiv (\sigalg{F}_t)_{t \geq 0}$ and $\filt{G}  \equiv (\sigalg{G}_t)_{t \geq 0}$, which are nested in the sense that $\filt{F} \subseteq \filt{G}$, i.e., $\sigalg{F}_t \subseteq \sigalg{G}_t$ holds for all $t \geq 0$. For a given process $X = (X_t)_{t \geq 0}$, we use $\filt{F}^X$ to denote the smallest right-continuous filtration that makes $X$ adapted. If $X$ is additionally nonnegative, let ${}^o X$ denote its $\filt{F}$-optional projection, which always exists but could take the value $\infty$; see, for example, Nikeghbali \cite[Theorem~4.1]{Nikeghbali:essay}. If $X$ is a semimartingale, we use $\Exp(X)$ to denote its stochastic exponential.

We consider an $\filt{F}$-adapted \emph{continuous}-path $\filt{G}$-semimartingale $S$, representing the price of a financial asset expressed in terms of a certain denomination. Everything that follows carries over to the case of a multi-asset case $S = (S^1, \cdots, S^d)$ for $d \in \mathbb{N}$, at the expense of more complicated notation; we refrain from considering multi-asset models as notation is already a bit heavy. However, we stress that continuity of the paths of $S$ will be important, as we shall explain at places.
All wealth is considered in terms of the same tradeable denomination, 
which is operationally the same as having an additional asset available 
with unit price.

The financial notions below can be considered under different filtrations, so we use $\filt{H}$ to generically denote either the ``small'' $\filt{F}$ or the ``large'' $\filt{G}$ filtration.

For given $x \geq 0$, let $\mathcal X^{\filt{H}} (x)$ denote the set of all \emph{nonnegative} wealth processes, i.e., all nonnegative  processes $V^{x,\theta}$ of the form
\begin{align*}
V^{x,\theta} = x + \int_0^\cdot \theta_u \ud S_u,
\end{align*}
where $\theta$ is $\filt{H}$-predictable and $S$-integrable. We set $\mathcal{X}^{\filt{H}} \dfn \bigcup_{x \geq 0} \mathcal{X}^{\filt{H}} (x)$.

For any $T > 0$, and $\xi \in L^0_+ (\sigalg{H}_T)$, we set
\[
x^{\filt{H}} (T, \xi) \dfn \inf \left \{ x \geq 0 : \, \exists V \in \mathcal{X}^{\filt{H}} (x) \text{ such that } V_T \geq \xi, \ \Pu \text{-a.e.} \right \}
\]
to be the hedging capital associated with $\xi$.

\begin{definition}
We shall say that the market is $\filt{H}$-\textbf{viable}, if $x^{\filt{H}} (T, \xi) = 0$ implies $\xi = 0$, $\Pu$-a.e., for any $T > 0$ and $\xi \in L^0_+ (\sigalg{H}_T)$.
\end{definition}

The concept of viability (for a specific filtration) is also known as absence of arbitrage of the first kind, or as the condition of locally, in time, no unbounded profit with bounded risk in Karatzas and Kardaras \cite{KK}.

\begin{definition}
An $\filt{H}$-\textbf{local martingale deflator} is a strictly positive $\filt{H}$-local martingale $Y$ such that $Y S$ is also a $\filt{H}$-local martingale. Correspondingly, an $\filt{H}$-supermartingale deflator is a strictly positive $\filt{H}$-supermartingale $Y$ such that $Y X$ is an $\filt{H}$-supermartingale for all $X \in \mathcal X$. \qed
\end{definition}

The class of all $\filt{H}$-local martingale deflators will be denoted by $\mathcal Y^{\filt{H}}$. 

\begin{theorem}[Choulli and Stricker \cite{Choulli:Stricker:1996}, Kardaras \cite{Kardaras_finitely}] \label{T:1FTAP}
	The following statements are equivalent:
	\begin{enumerate}
		\item\label{T:1FTAPi} The market is $\filt{H}$-viable.
		\item\label{T:1FTAPii} There exists an $\filt{H}$-local martingale deflator: $\mathcal Y^{\filt{H}} \neq \emptyset$. 
		\item\label{T:1FTAPiii} There exists an $\filt{H}$-supermartingale deflator.
		\item\label{T:1FTAPiv} Writing $S = A + M$, where $A$ is a continuous finite variation $\filt{H}$-adapted process and $M$ an $\filt{H}$-local martingale, it holds that $A = \int_0^\cdot H_u \ud [S, S]_u$, 
		where $H$ is an $\filt{H}$-predictable process such that the nondecreasing process $\int_0^\cdot H^2_u\ud [S, S]_u$ is $\Pu$-a.e.~finitely-valued.  
	\end{enumerate}  
\end{theorem}

Note that the \emph{structural condition} \eqref{T:1FTAPiv} in Theorem \ref{T:1FTAP} above implies that $S$-integrability of an $\filt{H}$-predictable process $\theta$ amounts to
\begin{equation} \label{eq:integr_under_NA1}
\int_0^\cdot \theta^2_u \ud [S, S]_u < \infty, \quad \Pu \text{-a.e.,}
\end{equation}
since the validity of \eqref{eq:integr_under_NA1} and the Cauchy-Schwartz inequality already imply that, $\Pu$-a.e.,
\begin{align*}
\int_0^\cdot | \theta_u | | \ud A_u | &\leq \int_0^\cdot | \theta_u | |H_u| \ud [S, S]_u \\
&\leq \left( \int_0^\cdot \theta^2_u \ud [S, S]_u \right)^{1/2} \left( \int_0^\cdot H^2_u \ud [S, S]_u \right)^{1/2} < \infty.
\end{align*}
In particular, under the structural condition \eqref{T:1FTAPiv}, $H$ is $S$-integrable, and we may define the specific $\filt{H}$-local martingale deflator $Y = 1 / \widehat{V}$, where
\begin{align} \label{eq:200826}
\widehat{V} \dfn \Exp \left( \int_0^\cdot H_u \ud S_u \right).
\end{align}
The above $\widehat{V}$ is a special wealth process in $\mathcal X^{\filt{H}} (1)$ called the $\filt{H}$-\textbf{num\'eraire}.

\begin{remark} \label{R:180525.1}
The nesting property $\filt{F} \subseteq {\filt{G}}$ seems to yield directly that
\begin{align}  \label{eq:180525.2}
\filt{G}\text{-viability implies } \filt{F}\text{-viability}.
\end{align}
However, the implication in \eqref{eq:180525.2} is a bit more subtle, the reason being that the inclusion $\X^{\filt{F}} \subseteq \X^{\filt{G}}$ is not in general true when $\filt{F} \subseteq {\filt{G}}$. Indeed, an $\filt{F}$-predictable process $\theta$ might be  $S$-integrable under $\filt{F}$, but not under $\filt{G}$. For example, assume that $\filt{F}$ is the natural filtration of a Brownian motion $W$ and that $\filt{G}$ is the smallest right-continuous filtration that makes $W$ adapted and $W_1$ a $\sigalg{G}_0$-measurable random variable; the process $\theta$ given by $\theta_t \dfn - 1 / \left( \sqrt{1-t} \log(1-t) \right) \indic_{\{t < 1\}}$, $t \in [0,1]$ is shown in Jeulin and Yor \cite{Jeulin:Yor:1979} to be $S$-integrable under $\filt{F}$, but not under $\filt{G}$; hence, in this example, $\X^{\filt{F}} \not \subseteq \X^{\filt{G}}$.
	
The previous remarks notwithstanding, $\filt{G}$-viability implies that $\X^{\filt{F}} \subseteq \X^{\filt{G}}$. Indeed, if an $\filt{F}$-predictable $\Real^d$-valued process $\theta$ is $S$-integrable, then it satisfies \eqref{eq:integr_under_NA1} \emph{a fortiori}, which then implies that it is also $S$-integrable in the filtration $\filt{G}$ in view of $\filt{G}$-viability and the discussion right before this Remark. Therefore, since $\filt{G}$-viability implies $\X^{\filt{F}} \subseteq \X^{\filt{G}}$, \eqref{eq:180525.2} follows.
\end{remark}

A wealth process $X \in \mathcal X^{\filt{H}}$ is called $\filt{H}$-\textbf{maximal} if, whenever $X' \in \mathcal X^{\filt{H}}$ is such that $X'_0 = X_0$ and $\Pu [X'_T \geq X_T] = 1$ for some $T \geq 0$, then, in fact, $\Pu [X'_T = X_T] = 1$.

\begin{definition}
The market $S$ is called $\filt{H}$-\textbf{complete} if, for any  $T > 0$ and $\xi \in L^0_+ (\sigalg{H}_T)$ with $x = x^{\filt{H}} (T, \xi) < \infty$, there exists a maximal $X \in \mathcal X^{\filt{H}} (x)$ such that $\Pu [X_T = \xi] = 1$.
\end{definition}

\begin{theorem}[Stricker and Yan \cite{Stricker:Yan}]
Assume $\filt{H}$-viability; equivalently, that $\mathcal Y^{\filt{H}} \neq \emptyset$. Then, the market is complete if, and only if, there exists only one $\filt{H}$-local martingale deflator.
\end{theorem}

\section{Projections of Local Martingale Deflators} \label{sec:proof1}

\subsection{A first result}

The main result of this section is the following.

\begin{theorem} \label{thm:1}
Let $Y$ be a  $\filt{G}$-local martingale deflator for $S$ with $\filt{F}$-optional projection ${}^o Y$.  Consider the multiplicative decomposition ${}^o Y = L (1 - K)$, where $L$ is an $\filt{F}$-local martingale and $K$ a nondecreasing $\filt{F}$-predictable $[0,1)$-valued process with $K_0 = 0$.
Then, $L$ is an $\filt{F}$-local martingale deflator for $S$.
\end{theorem}

Theorem~\ref{thm:1} will be immediate after the following two results have been established. Related to the first result is Gombani et al.~\cite[Proposition~3.1]{Gombani:Jaschke:Runggaldier}, where instead of local martingale deflators so-called linear price systems are considered.

\begin{proposition}  \label{P:180112.1}
Let $Y$ be a $\filt{G}$-supermartingale deflator for $S$. Then  its $\filt{F}$-optional projection ${}^o Y$ is a $\filt{F}$-supermartingale deflator for $S$.
\end{proposition}
\begin{proof}
For any $0 \leq s \leq t < \infty$ and $X \in \X^{\filt{F}} \subseteq \X^{\filt{G}}$, thanks to Remark~\ref{R:180525.1}, it holds that
\begin{align*}
\E \bra{{}^o Y_t X_t \such \sigalg{F}_s} &= \E \bra{ \E \bra{{}^o Y_t X_t \such \sigalg{F}_t} \such \sigalg{F}_s} \\ 
&= \E \bra{ \E \bra{Y_t X_t \such \sigalg{F}_t} \such \sigalg{F}_s} \\
&=\E \bra{Y_t X_t \such \sigalg{F}_s} \\
&= \E \bra{ \E \bra{Y_t X_t \such \sigalg{G}_s} \such \sigalg{F}_s} \\
&\leq \E \bra{ Y_s X_s \such \sigalg{F}_s} \\
&= {}^o Y_s X_s.
\end{align*}
It remains to show that ${}^o Y$ is strictly positive. For this, fix $t \geq 0$, and note that
\[
0 = \E \bra{{}^o Y_t \indic_{\{{}^o Y_t = 0\}}} = \E \bra{Y_t \indic_{\{{}^o Y_t = 0\}}}.
\]
Since $\Pu \bra{Y_t = 0} = 0$, it follows that $\Pu \bra{{}^o Y_t = 0} = 0$.
\end{proof}

The filtration in the statement of Proposition~\ref{prop:supmart_defl_to_local_mart_defl} below is implicit.

\begin{proposition} \label{prop:supmart_defl_to_local_mart_defl}
Consider the multiplicative decomposition $Y = L (1 - K)$ of a supermartingale deflator $Y$ for the continuous semimartingale $S$, where $L$ is a local martingale and $K$ a nondecreasing predictable $[0,1)$-valued process. Then, $L$ is a local martingale deflator for $S$.
\end{proposition}

\begin{proof}
Thanks to Theorem~\ref{T:1FTAP}, the process $\widehat{V}$ in \eqref{eq:200826} exists. Note that
$1/ \widehat{V}$ is a local martingale. By stopping, we may assume without
loss of generality, that $\widehat{V}$ is a uniformly integrable
martingale and hence defines a probability measure $\mathbb
Q$, equivalent to $\mathbb P$. Then $S$ is a local $\mathbb 
Q$--martingale and it suffices to prove that the local $\mathbb 
Q$--martingale $\widehat{V} L$ is a local martingale deflator for $S$ under 
$\mathbb Q$. Upon changing to $\mathbb Q$ and replacing $Y$ and $L$ by 
$\widehat{V} Y$ and $\widehat{V} L$, respectively, we may and shall assume that 
$S$ is a local $\mathbb P$-martingale in everything below.

Because $L$ is strictly positive, the Kunita-Watanabe decomposition yields the representation $L = \Exp(\int_0^\cdot \theta_u \ud S_u) N$, where $\theta$ is predictable and $S$-integrable, and $N$ is a strictly positive local martingale such that $[N, S] = 0$ holds. In order to prove the statement, it now suffices to show that $L = N$, i.e., $\int_0^\cdot \theta_u \ud S_u = 0$.

For each $\nin$, consider the positive wealth process $\Exp(n \int_0^\cdot \theta_u \ud S_u) \in \X$. Since
\[
Y \Exp\left(n \int_0^\cdot \theta_u \ud S_u \right) = (1 - K) \Exp\left(\int_0^\cdot \theta_u \ud S_u\right) \Exp\left(n \int_0^\cdot \theta_u \ud S_u \right) N
\]
is a supermartingale, $N$ is a strictly positive local martingale strongly orthogonal to the continuous semimartingale $S$ and $(1 - K) \Exp(\int_0^\cdot \theta_u \ud S_u) \Exp(n \int_0^\cdot \theta_u \ud S_u)$ is predictable, integration-by-parts implies that $(1 - K) \Exp(\int_0^\cdot \theta_u \ud S_u) \Exp(n \int_0^\cdot \theta_u \ud S_u)$ is also a (local) supermartingale. Write $1 - K = \Exp(- C)$, where $C$ is nondecreasing and predictable, with $\Delta C < 1$. Then,
\begin{align*} 
(1 - K) \Exp\left(\int_0^\cdot \theta_u \ud S_u\right) & \Exp\left(n \int_0^\cdot \theta_u \ud S_u\right) \\
&= \Exp(- C) \Exp\left(\int_0^\cdot \theta_u \ud S_u\right) \Exp\left(n \int_0^\cdot \theta_u \ud S_u\right) \\
&= \Exp\left(- C + (n+1) \int_0^\cdot \theta_u \ud S_u + n \int_0^\cdot \theta^2_u \ud [S, S ]_u \right)
\end{align*}
holds in view of Yor's formula, where we have used the fact that $[C, S] = 0$ holds. It follows that $- C + n \int_0^\cdot \theta^2_u \ud [S, S ]_u$ has to be a non-increasing process. Since this has to hold for all $\nin$, we obtain $\int_0^\cdot \theta^2_u \ud [S, S ]_u = 0$, which is the same as $\int_0^\cdot \theta_u \ud S_u = 0$.
\end{proof}

\subsection{Ramifications}
			
As mentioned after Theorem \ref{T:1FTAP}, a special $\filt{G}$-local martingale deflator is the one corresponding to the reciprocal of the $\filt{G}$-num\'eraire in  $\X^{\filt{G}}$. It is natural to ask whether the $\filt{F}$-optional projection of the reciprocal of the $\filt{G}$-num\'eraire is the reciprocal of the $\filt{F}$-num\'eraire. The following example shows that this is not necessarily the case, even if the reciprocal of the $\filt{G}$-num\'eraire is a $\filt{G}$-martingale. For a positive result in this direction, under additional assumptions, we refer to Proposition~\ref{P:180107.1} later on.

\begin{example} \label{Ex:180214.1}
Assume that the underlying probability space supports a standard Brownian motion $W$ and an independent  Bernoulli random variable $\Theta$ with $\Pu[\Theta = 1] = q = 1 - \Pu[\Theta = 0]$ for $q \in (0,1)$. The filtration $\filt{G}$ is given by $\sigalg{G}_t = \sigalg{F}_t^W \vee \sigma(\Theta)$ for all $t \geq 0$, while  $\filt{F} = \filt{F}^X = \filt{F}^S$, where
\begin{align*}
X \dfn  \Theta \int_0^\cdot \ud u  +  \int_0^\cdot (\indic_{u < 1} + \Theta \indic_{u \geq 1}) \ud W_u; \qquad 
S \dfn \Exp (X).
\end{align*}
Define the wealth process $\widehat{V}  \in \X^{\filt{G}}$ by
\begin{align*}
\widehat{V}  &\dfn \Exp\left(\int_0^\cdot \frac{\Theta}{S_u} \ud S_u\right)  = \Exp\left(\Theta^2  \int_0^\cdot \ud u  +  \Theta \int_0^\cdot (\indic_{u < 1} + \Theta \indic_{u \geq 1}) \ud W_u\right) \\
&= \Exp\left(\Theta \int_0^\cdot \left( \ud u  +  \ud W_u \right) \right),
\end{align*}
where we have used the fact that $\Theta = \Theta^2$ since $\Theta$ is $\{0,1\}$-valued. It is straightforward to check that $Y \dfn 1 / \widehat{V} = \Exp(- \Theta W)$ is a $\filt{G}$-local martingale deflator, which is obviously a $\filt{G}$-martingale. Moreover, $\Theta$ is $\sigalg{F}_1$-measurable, yielding $\sigalg{F}_t = \sigalg{G}_t$ for all $t \geq 1$ and
	\begin{align*}
		{}^o Y_t  = \E\left[\left.Y_t\right| \sigalg{F}_t\right] =
\indic_{\{\Theta = 0\}} + \e^{-X_t + t/2} \indic_{\{\Theta = 1\}},
\qquad t \geq 1.
	\end{align*}
	Straightforward computations give
	\[
	\Pu [\Theta = 0 \such  \sigalg{F}_t] = \frac{1-q}{1 - q + q \exp(X_t - t/2)},	\quad 0 \leq t < 1,
	\]
	implying that
	\[
	{}^o Y_t = \E\left[\left.\indic_{\{\Theta = 0\}} \right| \sigalg{F}_t\right] +  \e^{-X_t + t/2} \E\left[\left.\indic_{\{\Theta = 1\}} \right| \sigalg{F}_t\right]	= \frac{1}{1 - q + q \exp(X_t - t/2)},
	\]
	for $0 \leq t < 1$. Hence, ${}^o Y$ has a jump at time $t=1$.  Since the $\filt{F}$-num\'eraire has continuous paths, its reciprocal clearly cannot equal ${}^o Y$.
\end{example}

We now provide a result concerning the dynamics of $S$ in the smaller $\filt{F}$-filtration. To make headway, note that Theorem~\ref{T:1FTAP} yields some $\filt{G}$-predictable process $G$ and some $\filt{G}$-local martingale $M$ such that
\begin{align} \label{eq:180603.1}
	S = S_0 + \int_0^\cdot G_u \ud [M, M]_u + M.
\end{align}
\begin{proposition}
Assume that the $\filt{F}$-optional projection of $|G|$ is $(\Pu \times [S, S])$-a.e.~finite, i.e.
\begin{align} \label{eq:180603.2}
		\E \left[ \int_0^\infty \indic_{ \{{}^o |G|_u = \infty \}} \ud [S, S]_u \right] = 0.
\end{align}
Then, the $\filt{F}$-predictable projection $F$ of $G$ exists and satisfies $ \int_0^\cdot F_u^2 \ud [S, S]_u   < \infty$. Moreover,
\begin{align} \label{eq:180603.3}
	S = S_0 + \int_0^\cdot F_u \ud [S, S]_u + N,
\end{align}
where $N$ is an $\filt{F}$-local martingale.
\end{proposition}
\begin{proof}
	Without loss of generality, and upon using $\filt{F}$-localisation, we may assume that the $\filt{F}$-adapted processes $S$ and $[M, M] = [S,S]$ are uniformly bounded, hence $M$ is a $\filt{G}$-martingale.  An appropriate modification of Meyer \cite[Theoreme~$1'$]{Meyer:1973} yields that the dual optional $\filt{F}$-projection of $\int_0^\cdot G_u \ud [S, S]_u$ equals $ \int_0^\cdot F_u \ud [S,S]_u$. The fact that $ \int_0^\cdot F_u^2 \ud [S, S]_u   < \infty$ follows from Theorem \ref{T:1FTAP}, given that $\filt{G}$-viability implies $\filt{F}$-viability from \eqref{eq:180525.2}.  
\end{proof}

As the next example illustrates, although \eqref{eq:180603.3} always holds for some $\filt{F}$-predictable process $F$, the predictable $\filt{F}$-projection of $G$ does not need exist in general. (We are grateful to  Walter Schachermayer for proposing the idea for this example.)
\begin{example} \label{Ex:200826}
Let  $\Omega = \mathbb{N} \times C([0, \infty); \Real)$. Define $\Theta(\theta, w) = \theta$ and $W_t(\theta, w) = w_t$ for all $(\theta,w) \in \Omega$ and $t \in [0, \infty)$. Let $\filt{G}$ denote the smallest right-continuous filtration making $\Theta$ a $\sigalg{G}_0$-measurable random variable and $W$ adapted. Consider any probability measure $\mu$ on $2^\Natural$ with $\sum_{\theta \in \Natural}\theta \mu [ \{ \theta \}] = \infty$, and let $\Pu$ denote the product probability on $\sigalg{G}_\infty$ of $\mu$ and Wiener measure.  Note that $\E[\Theta] = \infty$ and that $\Theta$ and $W$ are independent under $\Pu$.

Theorem~1 in Prokaj and Schachermayer \cite{Prokaj:Schachermayer:2011} and a simple conditioning argument yield the existence of a $\filt{G}$-predictable process $H$ taking values in $\{-1, 1\}$, such that the process
\[
X \dfn \int_0^\cdot H_u \Theta \ud u + \int_0^\cdot H_u \ud W_u
\]
is an $\filt{F}$-Brownian motion, where $\filt{F} \dfn \filt{F}^X$. Note also that $\Theta$ and $X$ are independent under $\Pu$, again by a conditioning argument (indeed, the law of $X$ 
conditioned on $\Theta$ coincides with its unconditional one, namely the 
standard Wiener measure). Thus, for $u \geq 0$,
\begin{align}
\E\left[ |H_u \Theta|   \such \sigalg{F}_u\right] = \lim_{n \uparrow \infty} \E\left[ |H_u \Theta| \wedge n  \such \sigalg{F}_u\right] &= \lim_{n \uparrow \infty} \E\left[  \Theta \wedge n  \such \sigalg{F}_u\right] \nonumber\\
	&=  \lim_{n \uparrow \infty} \E\left[ \Theta \wedge n\right] = \E[\Theta] = \infty. \label{eq:200826.2}
\end{align}
Defining now $S \dfn \Exp(X)$, \eqref{eq:180603.1} holds with $M \dfn \int_0^\cdot S_u H_u \ud W$ and $G \dfn  H \Theta / S$. In particular, $\int_0^\cdot  G_u^2 \ud u = \Theta^2 \int_0^\cdot S_u^{-2} \ud u < \infty$ holds. Moreover it follows from \eqref{eq:200826.2} that \eqref{eq:180603.2} fails. Nevertheless, \eqref{eq:180603.3} holds with $F = 0$ and $N = \Exp(X)$; here, $F$ is not the predictable $\filt{F}$-projection of $G$, as the latter does not exist.
\end{example}

\begin{remark} \label{R:180526}  
If there exist finitely many $\filt{F}$-optional processes $(\Phi^i)_{i = 1, \cdots, I}$ for some $I \in \mathbb{N}$ such that $\sign(G) = \sign(\Phi^i)$, $[S, S]$-a.e., for some $i = 1, \cdots, I$, $\Pu$-a.e.~then   \eqref{eq:180603.2}  holds. To see this, note that by localisation we may assume uniform boundedness, say by a constant $\kappa > 0$, of the $\filt{F}$-measurable process $[S, S]$, as well as of all processes
\[
	\widetilde{S}^i  \dfn S_0 + \int_0^\cdot \sign(\Phi^i_u) \ud S_u = S_0 + \int_0^\cdot  \sign(\Phi^i_u) G_u \ud [M, M]_u + \int_0^\cdot \sign(\Phi^i_u)  \ud  M_u 
\]
for $i = 1, \cdots, I$. It now suffices to observe that
\begin{align*}
	\E\left[\int_0^\cdot |G_u| \ud [M, M]_u  \right]
		 &\leq \sum _{i = 1}^I \E\left[\indic_{\{\sign(G) = \sign(\Phi^i)\}} \left(\widetilde{S}^i - S_0  - \int_0^\cdot \sign(\Phi^i_u)  \ud  M_u \right)\right]\\
		 &\leq 2 I \kappa  +  \sum _{i = 1}^I \E\left[\left|\int_0^\cdot \sign(\Phi^i_u)  \ud  M_u \right|\right] < \infty;
\end{align*}
here the last inequality uses the fact that the stochastic integrals have bounded quadratic variation.
This then yields  \eqref{eq:180603.2}.

For example, assume that $\filt{G}$ supports a  Brownian motion $W$ and a $\sigalg{G}_0$-measurable $\Real$-valued random variable $\Theta$. Let $H$ denote any $\filt{G}$-predictable process such that $\int_0^\cdot H_u^2 \ud u < \infty$ and set
\[
	S \dfn \Exp\left(\Theta \int_0^\cdot H_u \ud u + W\right).
\]
Consider now a right-continuous filtration  $\filt{F}$ with  $\filt{F} \subseteq\filt{G}$ such that $S$ and $\sign(H)$ are $\filt{F}$-adapted. With $I = 2$, $\Psi^1 = \sign(H)$, and $\Psi^2 = -\sign(H)$,  the $\filt{F}$-optional projection of $|G|$, where $G \dfn H/S$, is $(\Pu \times [S,S])$-a.e.~finite, i.e., \eqref{eq:180603.2} holds. 
\end{remark}

\section{A Bayesian Framework}   \label{sec:Bayes}

\subsection{Set-up}
Consider some parameter space $\mathfrak{R}$, equipped with sigma algebra $\R$ and probability measure $\mu$, which will be the ``prior'' law of a parameter. Let $\Omega = \mathfrak{R} \times C([0, \infty); \Real)$. Define $\Theta(\theta, x) = \theta$ and $X_t(\theta, x) = x_t$ for all $(\theta,x) \in \Omega$ and $t \in [0, \infty)$. Define $\filt{F}$ as the smallest right-continuous filtration making $X$ adapted; i.e., $\filt{F} = \filt{F}^X$. Moreover, let $\filt{G}$ be the smallest right-continuous filtration containing $\filt{F}$ and further making $\Theta$ a $\sigalg{G}_0$-measurable random variable.  Next, define $\Qu$ as the product probability measure on $\sigalg{G}_\infty$ of $\mu$ and Wiener measure; under $\Qu$, $X$ is a $\filt{G}$-Brownian motion independent of $\Theta$, the latter random variable having law $\mu$. Also, under $\Qu$, $X$ is an $\filt{F}$-Brownian motion. Let $\Wu = \Qu |_{\sigalg{F}_\infty}$ denote Wiener measure.

Consider a functional $G : \Omega \times [0, \infty) \mapsto [- \infty, \infty]$, assumed to be $\filt{G}$-optional, which will serve as the drift functional for the stock returns in the filtration $\filt{G}$. We allow $G$ to take the values $\pm \infty$, although such values will not be ``seen'' by the solutions of the martingale problems we consider later. Define also $A : \Omega \times [0, \infty) \mapsto [0, \infty]$ via $A(\theta, x, \cdot) \dfn \int_0^\cdot G(\theta, x, u)^2 \ud u$; note that $A$ is non-decreasing in the time component.

For $\mu$-a.e.~$\theta \in \mathfrak{R}$, we assume the existence of a probability $\Pu^\theta$ on $\sigalg{F}_\infty$ such that $\Pu^\theta \ll_{\sigalg{F}_t} \Wu$ for all $t \geq 0$, $\int_0^\cdot |G(\theta, X, u)| \ud u$ is $\Pu^\theta$-a.e.~finitely valued and  $X - \int_0^\cdot G(\theta, X, u) \ud u$ is an $\filt{F}$-local $\Pu^\theta$-martingale.

Some remarks are in order. First of all, under the previous assumptions, the process $W^\theta \dfn X - \int_0^\cdot G(\theta, X, u) \ud u$ is actually an $(\filt{F}, \Pu^\theta)$-Brownian motion, as follows from L\'evy's characterisation theorem. Secondly, defining the set-valued process $\Sigma : \Omega \times [0, \infty) \mapsto \mathcal{R}$ via
\[
\Sigma_t \dfn \set{\theta \in \mathfrak{R} \such A(\theta, X, t) = \infty} \in \sigalg{F}_{t-}, \quad t \geq 0,
\]
Girsanov's theorem implies that
\begin{align*}
\zeta^\theta_t &\dfn \frac{\ud \Pu^\theta}{\ud \Wu}\Big|_{\sigalg{F}_t} \\
&= \exp \pare{\int_0^t G(\theta, X, u) \ud X_u - \frac{1}{2} A(\theta, X, t)} \indic_{\{\theta \notin \Sigma_t\}}, \quad (\theta,t) \in \mathfrak{R} \times [0, \infty),
\end{align*}
which, in particular, implies that $\Pu^\theta$ is necessarily unique. 
Thanks to Stricker and Yor \cite[Proposition~5]{Stricker:Yor} applied under $\Qu$ and a $\filt{G}$-localisation argument,  the mapping $\zeta : \mathfrak{R} \times C([0, \infty); \Real) \times [0, \infty) \mapsto [- \infty, \infty]$ may be chosen to be jointly measurable by taking an appropriate version. Finally, since $\Wu \bra{\theta \in \Sigma_t, \, \zeta^\theta_t > 0} = 0$, it follows that $\Pu^\theta \bra{\theta \in \Sigma_t} = 0$ holds for all $(\theta,t) \in \mathfrak{R} \times [0, \infty)$, even though  $\Qu \bra{\theta \in \Sigma_t} > 0$ is possible.

Define now $\Pu$ on $\sigalg{G}_\infty$ via $\Pu \bra{\ud \theta, \ud x} \dfn \mu \bra{\ud \theta} \Pu^\theta \bra{\ud x}$, and note that $W \dfn X - \int_0^\cdot G(\Theta, X, u) \ud u$ is a standard $(\filt{G}, \Pu)$-Brownian motion; in particular, $W$ and $\Theta$ are independent under $\Pu$.  Indeed, as in Example~\ref{Ex:200826}, this follows from the fact 
that the conditional law of $W$ 
given 
$\Theta$
coincides with its unconditional law.

In order to connect with the financial setting of the previous sections, one may define the asset price $S$ to equal to $X$ or, if one insists on positive asset prices, one may set $S = \Exp (X)$. Choosing one or the other is plainly a matter of interpretation, and will not affect the mathematical content of the discussion here. The fact that $\Pu^\theta \bra{\theta \in \Sigma_t} = 0$ for all $t \geq 0$, equivalent to finiteness of the process $A(\theta, x, \cdot) = \int_0^\cdot G(\theta, x, u)^2 \ud u$, implies by Theorem \ref{T:1FTAP} the $\filt{F}$-viability of the $\Pu^\theta$-model for all $\theta \in \mathfrak{R}$.

We are interested in the dynamics of $X$ on $\filt{F}$ under $\Pu$. For this, we  make one final assumption (recall also the discussion in Remark~\ref{R:180526}), namely
\begin{equation} \label{eq:180212}
\int_{\mathfrak{R}} 
|G(\theta, X, \cdot)| \zeta^\theta_\cdot \mu \bra{\ud \theta} < \infty, \quad (\Pu \times [X,X])\text{-a.e.}
\end{equation}
Under all the previous assumptions, Bayes' formula yields for $t \geq 0$ that
\[
\E_\Pu \bra{G(\Theta, X, t) \such \sigalg{F}_t} = \frac{\int_{\mathfrak{R}} G(\theta, X, t) \zeta^\theta_t \mu \bra{\ud \theta}}{\zeta_t}, \quad \text{where} \quad \zeta_t \dfn \int_{\mathfrak{R}} \zeta^\theta_t \mu \bra{\ud \theta}
\]
is an $(\filt{F}, \Qu)$-Brownian martingale. 
In fact, upon defining the random measure-valued process $(\mu_t)_{t \geq 0}$ via
\[
\frac{\mu_t [\ud \theta]}{\mu[\ud \theta]} = \frac{\zeta^\theta_t}{\zeta_t} \equiv \frac{\zeta(\theta, X, t)}{\int_{\mathfrak{R}} \zeta(\eta, X, t) \mu \bra{\ud \eta}},
\]
it follows that $\E_\Pu \bra{G(\Theta, X, t) \such \sigalg{F}_t} = \int_{\mathfrak{R}} G(\theta, X, t) \mu_t \bra{\ud \theta}$.
Therefore, defining the functional $F: C([0, \infty); \Real) \times [0, \infty) \rightarrow (- \infty, \infty]$ via
\[
F(x, t) \dfn \frac{\int_{\mathfrak{R}} G(\theta, x, t) \zeta(\theta, x, t) \mu \bra{\ud \theta}}{\int_{{\mathfrak{R}}} \zeta(\theta, x, t) \mu \bra{\ud \theta}},  \text{ if } \int_{\mathfrak{R}} \abs{G(\theta, x, t)} \zeta(\theta, x, t) \mu \bra{\ud \theta} < \infty,
\]
and $F(x, t) \dfn \infty$ otherwise, it follows that $W^{\filt{F}} \dfn X - \int_0^\cdot F(X, u) \ud u$ is $(\filt{F}, \Pu)$-Brownian motion. 

The  $\filt{G}$-local martingale deflators for $S = X$ (or $S = \Exp(X)$) are of the form
\begin{align*}
Y &= h(\Theta) \exp \pare{ - \int_0^\cdot G(\Theta, X, u) \ud W_u - \frac{1}{2}  \int_0^\cdot G(\Theta, X, u)^2 \ud u} \\
&= h(\Theta) \exp \pare{ - \int_0^\cdot G(\Theta, X, u) \ud X_u + \frac{1}{2}  \int_0^\cdot G(\Theta, X, u)^2 \ud u} = h(\Theta) \frac{1}{\zeta^\Theta},
\end{align*}
where $h: {\mathfrak{R}} \rightarrow (0,\infty)$ is any strictly positive Borel function with the property $\int_{\mathfrak{R}} h(\theta) \mu[\ud \theta] = 1$. Note that since $\Pu \bra{\Theta \in \Sigma_t} = 0$, we have $\Pu \bra{\zeta^\Theta_t > 0} = 1$ for all $t \geq 0$. The optional projection of any such $Y$ on $\filt{F}$ satisfies, by Bayes' rule,
\begin{align*}
{}^o Y_t &= \E_\Pu \bra{Y_t \such \sigalg{F}_t} = \frac{\int_{\mathfrak{R}} \pare{h(\theta) / \zeta^\theta_t} \zeta^\theta_t \indic_{\{\zeta^\theta_t > 0\}} \mu [\ud \theta] }{\zeta_t} = \frac{\int_{\mathfrak{R}} h(\theta) \indic_{\{\zeta^\theta_t > 0\}} \mu [\ud \theta] }{\zeta_t} \\
&= (1 - K_t^h) \frac{1}{\zeta_t}, \quad t \geq 0,
\end{align*}
where
\[
K^h_t \dfn \int_{\Sigma_t} h(\theta) \mu [\ud \theta]\qquad t \geq 0.
\]
Note that $K^h$ is a nondecreasing, $\filt{F}$-predictable process. In particular, there is ``loss of mass'' exactly when certain models become impossible. If the conditional law of $\Theta$ under $\Pu$ given $\sigalg{F}_t$ maintains the same support as $\Theta$ has for all $t \geq 0$, it follows that $K^h = 0$. 
By Theorem~\ref{thm:1}, $1/\zeta$ is an $\filt{F}$-local martingale deflator.  This uses the fact that  $1/\zeta$ is indeed an $(\filt{F}, \Pu)$-local martingale since $\zeta$ is an $(\filt{F}, \Qu)$-Brownian martingale, hence continuous, not jumping to zero.

\begin{lemma}  \label{L:190824}
It holds that $\set{\zeta = 0} = \set{\mu[\Sigma] = 1}$. In particular, for any  $\sigalg{F}_t$-measurable nonnegative $\xi$ and $t \geq 0$, 
it holds that
\[
\E_\Pu \bra{\frac{1}{\zeta_t} \xi} = \E_\Wu \bra{\xi \indic_{\{\mu[\Sigma_t] < 1\}}}.
\]
\end{lemma}

\begin{proof}
Simply note that
\[
\set{\zeta = 0} = \set{ \int_{\mathfrak{R}} \zeta^\theta \mu [\ud \theta] = 0} = \set{ \mu [\Sigma] = 1}.
\]
Given that $\zeta$ is the density process of $\Pu$ with respect to $\Wu$ on $\filt{F}$, we have 
\[
\E_\Pu \bra{\frac{1}{\zeta_t} \xi} = \E_\Wu \bra{\xi \indic_{\{\zeta_t>  0\}}}, \qquad t \geq 0,
\]
which immediately gives the result.
\end{proof}

As a corollary of the above, we obtain that the ``default'' (in the terminology of Elworthy et al.~\cite{Elworthy_Li_Yor_99}) of the $(\filt{F}, \Pu)$-local martingale $1/\zeta$ equals
\[
1 - \E_\Pu \bra{\frac{1}{\zeta_t}} = 1 - \Wu \bra{ \mu[\Sigma_t] < 1 } = \Wu \bra{ \mu[\Sigma_t] = 1 }, \qquad t \geq 0.
\]
This is clear: $1/\zeta$ will be an $(\filt{F}, \Pu$)-martingale if and only if $\Pu$ and $\Wu$ are locally equivalent, which will happen exactly when, under $\Wu$,  the family of parameters that yield a strictly positive Radon–Nikodym derivative at time $t$ has strictly positive $\mu$-measure, for each $t \geq 0$.

\begin{remark} \label{R:190824}
The $(\filt{F}, \Pu)$-market is complete. Indeed, fix some $T \geq 0$ and some nonnegative $\sigalg{F}_T$-measurable random variable $D_T$ with $D_0 = \E_\Pu \bra{D_T/\zeta_T}  < \infty$.  We then also have $D_0 = \E_\Qu [D_T \indic_{\{\zeta_T > 0\}}] < \infty$ and the martingale representation theorem gives the existence of an $(\Qu, S)$-integrable $H$ such that $D_0 + \int_0^T H_u \ud S_u = D_T \indic_{\{\zeta_T > 0\}}$ holds $\Qu$-a.e. This also implies that  $D_0 + \int_0^T H_u \ud S_u = D_T$ holds $\Pu$-a.e.  
\end{remark}

\begin{example} \label{ex:3.3}
Let $\mu$ be an arbitrary law on $\mathfrak{R} \dfn \mathbb{R}$, and  set $G(\cdot, \cdot, \theta) \dfn \theta H(\cdot, \cdot)$ for all $\theta \in \Real$, where $H$ is $\filt{F}$-optional with $0 < \int_0^t H^2(u, x) \ud u < \infty$ for all $t > 0$ and $x \in  C([0, \infty); \Real)$. Then, 
\begin{align*}
	\zeta^\Theta &= \Exp\left(\Theta \int_0^\cdot H(u, X) \ud X_u \right) \\
		&=  \exp\left(\Theta \int_0^\cdot H(u, X) \ud X_u - \frac{\Theta^2}{2} \int_0^\cdot H^2(u, X)  \ud u\right)
\end{align*}
 holds; hence  $\Sigma = \emptyset$ in this case.
Moreover, for $t > 0$, we have
\[
\E_\Pu[|\Theta| |\sigalg{F}_t] = \frac{\int_{\mathbb{R}} |\theta| \exp(\theta \int_0^t H(u, X) \ud X_u   - \theta^2 /{2}  \int_0^t H^2(u, X)  \ud u) \mu(\ud \theta)} {\zeta_t} < \infty,
\]
where 
\[
	\zeta = \int_{\mathbb{R}} \exp\left(\theta \int_0^\cdot H(u, X) \ud X_u   - \frac{\theta^2}{2}  \int_0^\cdot H^2(u, X)  \ud u\right) \mu(\ud \theta).
\]
Hence, all the assumptions of the present section, including \eqref{eq:180212}, are satisfied and
\begin{align*}
	X = \Theta  \int_0^\cdot H(u, X) \ud u +  W =  \int_0^\cdot \E_\Pu[\Theta|\sigalg{F}_u] H(u, X) \ \ud u +  W^{\filt{F}}
\end{align*}
for some $(\filt{G}, \Pu)$-Brownian motion $W$ and some   $(\filt{F}, \Pu)$-Brownian motion $W^{\filt{F}}$. However heavy-tailed the law of $\Theta$ may be (and even if it does not have any moments), its generalised conditional expectations given $\sigalg{F}_\cdot$ exist. This example with $H = 1$ is discussed in Kailath \cite{Kailath:1971}; see also Remark~\ref{R:180526}.
\end{example}

\begin{example}  \label{Ex:190823}
Let $\mu$ be an arbitrary law on $\Real$, and $G(\theta, x, t) \dfn  - \indic_{{\theta x_t < 1}}  \theta  / (1- \theta x_t)$  whenever $\theta \in \Real$, $t \geq 0$ and $x \in C([0,\infty), \Real)$. The corresponding dynamics
\begin{equation} \label{eq:dynam_cond_never cross}
X =  - \int_0^\cdot \frac{\Theta}{1-\Theta X_u } \ud u +  W
\end{equation}
are that of Brownian motion conditioned that it never crosses the level $1 / \Theta$, with the case $\Theta = 0$ simply corresponding to Brownian motion. For future reference, note that
\begin{align}  \label{eq:180213.1}
	\lim_{t \uparrow \infty} X_t = -\infty, \quad \text{on }\{\Theta > 0\}, \quad \Pu\text{-a.e.},
\end{align}
and, correspondingly, $\lim_{t \uparrow \infty} X_t = \infty$ on $\{\Theta < 0\}$, $\Pu$-a.e.

In this example, one easily computes $\zeta^\theta = (1 - \theta X) \indic_{(\underline{\theta}, \, \overline{\theta}) } (\theta)$ for each $\theta \in \Real$, where we define $\underline{\theta} := 1 / \inf_{u \in [0, \cdot]} X_u < 0$ and $\overline{\theta} := 1 / \sup_{u \in [0, \cdot]} X_u > 0$. Clearly, $\Sigma = \Real \setminus (\underline{\theta}, \, \overline{\theta})$. We compute  
\[
\zeta_t =  \int_{(\underline{\theta}_t, \overline{\theta}_t)} \pare{1 - \theta X_t } \mu \bra{\ud \theta} = \mu \bra{(\underline{\theta}_t, \overline{\theta}_t)} - X_t \int_{(\underline{\theta}_t, \overline{\theta}_t)}  \theta \mu \bra{\ud \theta}, \qquad t \geq 0.
\]
Note also that
\begin{align*}
\int_0^\infty G(\theta, X, t) \zeta^\theta_t \mu \bra{\ud \theta} &= - \int_{(\underline{\theta}_t, \overline{\theta}_t)} \theta (1-\theta X_t )^{-1} \pare{1 - \theta X_t }  \mu \bra{\ud \theta} \\
&= - \int_{(\underline{\theta}_t, \overline{\theta}_t)}  \theta \mu \bra{\ud \theta}, \qquad t \geq 0.
\end{align*}
Defining
\[
\widehat{\Theta}_t = \frac{1}{\mu \bra{(\underline{\theta}_t, \overline{\theta}_t)}} \int_{(\underline{\theta}_t, \overline{\theta}_t)}  \theta \mu \bra{\ud \theta}, \qquad t > 0,
\]
observe that
\[
F(X, t) = \frac{- \int_{(\underline{\theta}_t, \overline{\theta}_t)}  \theta \mu \bra{\ud \theta}}{\mu \bra{(\underline{\theta}_t, \overline{\theta}_t)} - X_t \int_{(\underline{\theta}_t, \overline{\theta}_t)}  \theta \mu \bra{\ud \theta}} = - \frac{\widehat{\Theta}_t}{1-\widehat{\Theta}_t X_t }.
\]
It follows that that dynamics of $X$ under $\filt{F}$ are
\[
 X = - \int_0^\cdot \frac{\widehat{\Theta}_u}{1 - \widehat{\Theta}_u X_u } \ud u +  W^{\filt{F}},
\]
which are the same dynamics as \eqref{eq:dynam_cond_never cross} with $\Theta$ there replaced by the process $\widehat{\Theta}$.

Note that $K^1 = 1 - \mu[(\underline{\theta}, \overline{\theta})]$. 
In this example, $1/\zeta$ will be an actual $(\filt{F}, \Pu)$-martingale if and only if $\mu [(\underline{\theta}_t, \overline{\theta}_t)] > 0$ holds $\Wu$-a.e., for all $t \geq 0$, which is equivalent to saying that $\mu [(-\varepsilon, \varepsilon) ] > 0$ holds for all $\varepsilon > 0$.

Let us also note that the distribution of the overall maximum $X^*_\infty \dfn \max_{t \geq 0} X_t$ can be computed in this setup. To this end, fix $y > 0$ and recall from \eqref{eq:180213.1} that $\Pu^\theta[X^*_\infty > y] = 1$ if $\theta \leq 0$.	If $\theta > 0$ then
	\begin{align*}
	\Pu^\theta[X^*_\infty > y] &= \Pu^\theta\left[\min_{t \geq 0} \zeta^\theta_t < 1  - \theta {y}\right] = \Pu^\theta\left[\max_{t \geq 0} \frac{1}{\zeta^\theta_t} \geq \frac{1}{ 1  - \theta y} \right]  =  (1 - \theta y)_+.
	\end{align*}
	Here we used the facts that the $(\filt{F}, \Pu^\theta)$-local martingale ${1}/{\zeta^\theta}$ satisfies  ${1}/{\zeta^\theta_\infty} = 0$ for each $\theta > 0$.  
	Hence we get
	\begin{align} \label{eq:180214.1}
	\Pu[X^*_\infty > y] &= \mu\left[\left(-\infty, \frac{1}{y}\right)\right]  - y \int_0^{1/y}  \theta \mu \bra{\ud \theta}. 
	\end{align}
Similar computations  hold also for the overall minimum of $X$.	
\end{example}

\begin{remark}
	Explicit formulas for the quantities in Example~\ref{Ex:190823} may be obtained for nice laws $\mu$. For example, if $\mu [\ud \theta] =  \indic_{\theta > 0} \theta^{-3} \e^{-1/\theta}  \ud \theta$ for all $\theta \in \Real$ (inverse Gamma distribution), one obtains 
	\[
		\frac{1}{\mu\left[\left(0, 1 / x^*\right)\right]}  \int_0^{1/x^*} \theta \mu[\ud \theta] =   \frac{1}{\int_{x^*}^\infty  u \e^{-u} \ud u} \e^{-x^*} = \frac{1}{1+x^*}, \qquad x^* > 0.
	\]
	This then yields $\widehat \Theta = 1/(1+X^*)$, where $X^* \dfn \max_{u \in [0,\cdot]} X_u$ is the running maximum of $X$, hence
	$
	F(X, t) = - {1}/{(1 +  X^*_t - X_t)}
	$
	for all $t \geq 0$.
	We thus obtain
	\[
	 X = - \int_0^\cdot \frac{1}{1 +  X^*_u - X_u} \ud u+  W^{\filt{F}}
	\]
	for some $(\filt{F}, \Pu)$-Brownian motion $W^{\filt{F}}$.
	Furthermore, 
	\[
	\zeta = \mu \bra{\left(0, \frac{1}{X^*}\right)} - X \int_0^{1/X^*} \theta \mu \bra{\ud \theta} = (1 + X^* - X) \e^{- X^*},
	\]
	giving, in conjunction with \eqref{eq:180213.1}, that the limiting conditional law for $\Theta$ is
	\begin{align*}
	\mu_\infty [\ud \theta] &= \pare {\lim_{t \uparrow \infty} \frac{\zeta^\theta_t}{\zeta_t} } \mu [\ud \theta] = \theta \e^{X^*_\infty} \indic_{\{1/\theta > X^*_\infty\}} \mu [\ud \theta] \\
	&= \frac{1}{\theta^2}
	\e^{ - (1/\theta - X^*_\infty)} \indic_{\{1/\theta > X^*_\infty\}} \ud \theta;
	\end{align*}
	in other words, $1/\Theta - X^*_\infty$ given $\sigalg{F}_\infty$ has the standard exponential law under $\Pu$.
	Moreover, \eqref{eq:180214.1} yields that $X^*_\infty$ has also the standard exponential law under $\Pu$.  Hence $1/\Theta$ is the sum of the two independent standard exponentially distributed random variables $1/\Theta - X^*_\infty$ and $X^*_\infty$.  Note also that the overall maximum $X^*_\infty$ of $X$ has the same distribution as the overall maximum of Brownian motion with drift rate $-{1}/{2}$; see, for example, Karatzas and Shreve \cite[Exercise~3.5.9]{KS1}.
\end{remark}

\begin{remark}
	Fix $t > 0$ and an $\sigalg{F}_t$-measurable nonnegative random variable $\xi$, representing the payoff of a contingent claim.
	As already observed in Remark~\ref{R:190824}, the   $(\filt{F}, \Pu)$-market is complete. Indeed, the price $p$ of $\xi$  in  the   $(\filt{F}, \Pu)$-market
	equals
	\[
		p = \E_{\Pu} \left[\frac{1}{\zeta_t} \xi\right] =   \E_\Wu \bra{\xi \indic_{\{\mu[\Sigma_t] < 1\}}}
	\]
	by Lemma~\ref{L:190824}. Similarly, in the $(\filt{G}, \Pu)$-market, one has the $\sigalg{G}_0$-measurable price $p^\Theta$, where
	\[
	p^\theta \dfn \E_\Wu \big[ \xi \indic_{\{ \zeta^\theta_t > 0\}} \big] = \E_\Wu \bra{\xi \indic_{\{\theta \notin \Sigma_t\}}}, \qquad \theta \in \mathfrak{R}.
	\]
 	It is clear, both by economic and by mathematical reasoning, that $p^\Theta \leq p$, $\Pu$-a.e. 
	
	Let us now consider the question how $p$ and $\Pu\rm{-}\ess \sup p^\Theta$ relate (that is, how does the hedging cost of an ``uninformed'' agent relate to the worst-case hedging cost of an ``informed'' agent) in the context of Example~\ref{Ex:190823}. Using the fact that $\Sigma = \Real \setminus (\underline{\theta}, \, \overline{\theta})$, we have
	\[
	p = \E_\Wu \bra{\xi \indic_{\{ \mu[(\underline{\theta}_t, \, \overline{\theta}_t)\}] > 0}}; \qquad p^\theta = \E_\Wu \bra{\xi \indic_{\{\underline{\theta}_t < \theta < \overline{\theta}_t\}}}, \quad \theta \in \Real.	
	\]
	First, note that in the three cases $\mu[[0,\infty)] = 1$, $\mu[(-\infty, 0]] = 1$, or $\mu [(-\varepsilon, \varepsilon) ] > 0$  for all $\varepsilon > 0$, we have $p = \Pu\rm{-}\ess \sup p^\Theta$.  In words, in these three cases the worst-case hedging cost of the informed agent equals the hedging cost of the uniformed agent.   Indeed, if  $\mu[[0,\infty)] = 1$ then 
	\[
	\Pu\rm{-}\ess\sup p^\Theta =  \E_\Wu \bra{\xi \indic_{\{  (\Pu\rm{-}\ess\inf \Theta)  < \overline{\theta}_t\}}} =  \E_\Wu \bra{\xi \indic_{\{\mu[[0, \overline{\theta}_t)] > 0\}}} = p,
	\]
	the case $\mu[(-\infty, 0])] = 1$ is symmetric, and if $\mu [(-\varepsilon, \varepsilon) ] > 0$  for all $\varepsilon > 0$ holds then $\Pu\rm{-}\ess\sup p^\Theta =  \E_\Wu \bra{\xi} = p$.
	
	Consider now the complementary case where there exist $\varepsilon_1 > 0$, $\varepsilon_2 > 0$ with
	$\mu[(-\varepsilon_1, \varepsilon_2)] = 0$ and  $\mu[(-\varepsilon_1-\varepsilon, -\varepsilon_1]] > 0$ and $\mu[[\varepsilon_2, \varepsilon_2 + \varepsilon)] > 0$ for all $\varepsilon > 0$.
	For the unit claim $\xi \equiv 1$ we then have
	\[
	p = \Wu \bra{\underline{\theta}_t \leq - \varepsilon_1 \text{ or } \overline{\theta}_t \geq \varepsilon_2} >  \Wu \bra{\, \underline{\theta}_t \leq - \varepsilon_1 } \vee \Wu \bra{\, \overline{\theta}_t \geq \varepsilon_2} = \Pu\rm{-}\ess\sup p^\Theta.
	\]
	Therefore, even the worst-case hedging cost of the informed agent is strictly smaller than the uninformed agent's hedging cost. Note that, in all cases, the replication strategy for the informed agent starting from $p^\Theta$ depends on $\Theta$. In the case $\Pu\rm{-}\ess \sup p^\Theta < p$, the superreplication strategy of the informed agent starting from deterministic amount $\Pu\rm{-} \ess \sup p^\Theta$ also depends on $\Theta$; however, when $\Pu\rm{-}\ess \sup p^\Theta = p$, no knowledge of $\Theta$ is required in order to (super)replicate starting from $p$.
\end{remark}

\section{Under the Presence of a Dominating Measure} \label{sec:dominating}

We now consider a more general setup than in Sect.~\ref{sec:Bayes}.
We assume throughout this section the existence of a $\filt{G}$-local martingale deflator $Y$.
 Moreover, we make the following assumption.

\begin{ass}  \label{Ass00}
There exist a probability measure $\Qu$ and a $(\filt{G}, \Qu)$-martingale  $Z$  such that $(\mathrm d \Pu / \mathrm d \Qu)|_{\sigalg{G}_t}= Z_t$ for all $t \geq 0$ and  $Z = 1 / Y$, $\Pu$-a.e. \qed
\end{ass}
 We refer to F\"ollmer \cite{F1972} and Perkowski and Ruf \cite{Perkowski_Ruf_2014} for sufficient conditions for the existence of such a probability measure $\Qu$ and process $Z$. Note, in particular, that $\Pu \ll_{\mathcal{F}_t} \Qu$ holds for all $t \geq 0$.

 In the sequel, we shall need to consider optional projections under both probabilities $\Pu$ and $\Qu$; therefore, for the purposes of this section, we shall denote explicitly, via a superscript, the probability under which the projection is considered.
 
 Under Assumption~\ref{Ass00}, Bayes' rule yields
\begin{align}
{}^o Y^{\Pu}_t &= \E_\Pu \bra{\left.Y_t \right| \sigalg{F}_t} = 
\frac{\E_\Qu\left[\left.Y_t Z_t \indic_{\{Z_t>0\}} \right| \sigalg{F}_t\right] }
{\E_\Qu\left[\left. Z_t  \right| \sigalg{F}_t\right]} \nonumber \\
&= 
\frac{\Qu[Z_t>0 |  \sigalg{F}_t]}{{}^o Z_t^{\Qu} } = (1 - K_t) M_t \frac{1}{{}^o Z_t^{\Qu}}, \quad t \geq 0,  \label{eq:171228'}
\end{align}
where 
\begin{align}  \label{eq:180101.1}
	{}^o Z_t^{\Qu}= \E_\Qu\left[\left. Z_t  \right| \sigalg{F}_t\right], \qquad t \geq 0,
\end{align}
and $(1-K) M$ is the multiplicative Doob-Meyer decomposition (see, for example, \cite[Proposition~B.1]{Perkowski_Ruf_2014}) of the $(\filt{F}, \Qu)$-supermartingale $\Qu[Z_\cdot >0 \, |  \, \sigalg{F}_\cdot]$; i.e., $K$ is a nondecreasing  $\filt{F}$-predictable $[0,1]$-valued process with $K_0 = 0$ and $M$ is an $(\filt{F}, \Qu)$-local martingale
with
\begin{align} \label{eq:180101.2}
	(1-K_t) M_t =  \Qu\left[\left.Z_t > 0 \right|  \sigalg{F}_t\right] =  \Qu\left[\left.\tau_0 > t \right|  \sigalg{F}_t\right], \qquad t \geq 0,
\end{align}
where we have introduced the $\filt{G}$-stopping time
\begin{align*} 
	\tau_0 &\dfn \inf\{t \geq 0 :  Z_t = 0\}.
\end{align*}
To ensure uniqueness of the multiplicative decomposition we assume that $M = M^{\rho}$ and $K = K^{\rho}$, where $\rho$ is the first time that $\Qu\left[\left.Z_\cdot > 0 \right|  \sigalg{F}_\cdot\right]$ hits zero, and additionally that $\Delta M_\rho = 0$ on the event $\{K_\rho = 1\}$. 

Note that the Bayesian setup of Sect.~\ref{sec:Bayes} leads to $Z = \zeta^\Theta$ and $M = 1$.

Let us collect some properties on these processes that we have introduced so far.

\begin{proposition} \label{P:180114.1}
	In the notation of this section, and under Assumption~\ref{Ass00}, the following statements hold.
	\begin{enumerate}
		\item\label{P:180114.1.i} The  process  $1 / {}^o Z^{\Qu}$ is an $(\filt{F}, \Pu)$-supermartingale, and satisfies
\begin{align} \label{eq:180114.1}
	\E_\Pu\left[\frac{1}{{}^o Z_t^{\Qu}} \indic_{A}\right] = \Qu\left[\{{}^o Z_t^{\Qu}  > 0\} \cap A\right], \quad t \geq 0, \quad A \in \sigalg{F}_t. 
\end{align}		
		\item\label{P:180114.1.ii} The process $ M / {}^o Z^{\Qu}$ is an  $(\filt{F}, \Pu)$-local martingale. Hence, the right-hand-side of \eqref{eq:171228'} also leads to the multiplicative Doob-Meyer decomposition of the $(\filt{F}, \Pu)$-supermartingale ${}^o Y^{\Pu}$.  
	\end{enumerate}
\end{proposition}
\begin{proof}
Thanks to $(\ud \Pu / \ud \Qu)|_{\mathcal{F}_t} = {}^o Z_t^{\Qu}$ we have \eqref{eq:180114.1}, which then yields the statement in \eqref{P:180114.1.i}.
	
	Fix now $s, t \geq 0$ with $s<t$ and $A \in \sigalg{F}_s$ and let $\tau$ denote any bounded $\filt{F}$-stopping time such that $M^{\tau}$ is an $(\filt{F}, \Qu)$-martingale and ${}^o Z^{\Qu}$ is uniformly bounded away from zero on $\dbraco{0, \tau}$. Since ${}^o Z^{\Qu}$ is an $(\filt{F}, \Qu)$-martingale we then have $\{{}^o Z^{\Qu}_\tau = 0\} \subseteq\{M_\tau = 0\}$. Hence, from \eqref{eq:180114.1} we get
\begin{align*}
	\E_\Pu\left[\frac{M_t^\tau}{({}^o Z_t^{\Qu})^\tau} \indic_{A}\right] &= \E_\Qu\left[M_t^\tau \indic_{\{Z_{t \wedge \tau}^{\Qu}  > 0\}} \indic_{A}\right]
		=\E_\Qu\left[ M_t^\tau \indic_{A}\right] = \E_\Qu\left[ M_s^\tau \indic_{A}\right]   \\
		&= \E_\Qu\left[ M_s^\tau \indic_{\{{}^o Z^{\Qu}_{s \wedge \tau} > 0\}} \indic_{A}\right] = \E_\Pu\left[ \frac{M_s^\tau}{({}^o Z^{\Qu}_s)^\tau} \indic_{A}\right].
\end{align*}	
Hence $M^{\tau} / ({}^o Z^{\Qu})^\tau$ is an $(\filt{F}, \Pu)$-martingale.

Let now $(\tau_n')_{n \in \Natural}$ denote an $(\filt{F}, \Qu)$-localisation sequence of $M$ and let   $\tau_n''$ denote the first time that ${}^o Z^{\Qu}$ crosses the level $1/n$, for each $n \in \Natural$. Defining now $\tau_n = \tau_n' \wedge \tau_n''$ for each $n \in \Natural$ we get $\lim_{n \uparrow \infty} \tau_n = \infty$, $\Pu$-a.s, $M^{\tau_n}$ is an  $(\filt{F}, \Qu)$-martingale, and ${}^o Z^{\Qu}$ is uniformly bounded away from zero on $\dbraco{0, \tau_n}$. This then yields  statement~\eqref{P:180114.1.ii}.
\end{proof}

\begin{proposition} \label{P:180101.1}
	In the notation of this section, and under Assumption~\ref{Ass00}, the following statements concerning the optional projection ${}^\circ Y^\Pu$ are equivalent:
	\begin{enumerate}
		\item\label{P:180101.1.i} ${}^\circ Y^\Pu$ is an  $(\filt{F}, \Pu)$-local martingale.
		\item\label{P:180101.1.ii} $K$ is $\{0, 1\}$-valued, $\Qu$-a.e.
	\end{enumerate}
	Under any of the above equivalent conditions, it holds that $M = 1$, $\Qu$-a.e.; hence also ${}^\circ Y^\Pu = 1 / {}^o Z^{\Qu}$, $\Pu$-a.e.
\end{proposition}

\begin{proof}
	Let us first assume that  statement \eqref{P:180101.1.i}   holds, i.e., ${}^\circ Y^\Pu = (1-K)M/{}^\circ Z^\Qu$ is an  $(\filt{F}, \Pu)$-local martingale. Then Proposition \ref{P:180114.1}\eqref{P:180114.1.ii} yields that $K = 0$, $\Pu$-a.e; hence 
		 \begin{align} \label{eq:190912}
	 	\{K > 0\} \subseteq \{^\circ Z^\Qu = 0\} =  \{K = 1\} \cup \{M = 0\}, \qquad \text{$\Qu$-a.e.}
	\end{align}
	Furthermore, since the $(\filt{F}, \Qu)$-supermartingale $\Qu[Z_\cdot >0 \, |  \, \sigalg{F}_\cdot] = (1 - K) M$ is $[0,1]$-valued, we have 
\begin{align}  \label{eq:180101.3}
	 	\{K = 0\} \subseteq\{M \leq 1\}, \qquad \text{$\Qu$-a.e.}
\end{align}	
Combining now \eqref{eq:190912} and \eqref{eq:180101.3} yields that $M \leq 1$ on $\dbraco{0 , \rho}$, where $\rho$ is the predictable time when $K$ hits one. Since additionally by assumption $\Delta M_\rho = 0$ and $M = M^\rho$ on the event $\{K_\rho = 1\}$, we have $M \leq 1$. Next, since $M$ is also an $(\filt{F}, \Qu)$-local martingale with $M_0 = 1$ we obtain that $M = 1$, $\Qu$-a.e.  Then again recalling \eqref{eq:190912} yields that  $K$ is $\{0, 1\}$-valued, $\Qu$-a.e.

Assume now that statement~\eqref{P:180101.1.ii} holds. Then since $K < 1$ holds $\Pu$-a.e., we have $K = 0$, $\Pu$-a.e., and an application of Lemma~\ref{P:180114.1}\eqref{P:180114.1.ii} yields  \eqref{P:180101.1.i}.
\end{proof}

\begin{example}
Assume that the underlying probability space, equipped with the probability measure $\Qu$, supports a $\Qu$-Brownian motion $W$ and an independent  $\Real$-valued random variable $\Theta$.  Consider the filtration $\filt{G}$ to be the smallest right-continuous one that makes $W$ adapted, and such that $\Theta$ is $\mathcal{G}_0$-measurable. Moreover, consider the $\filt{G}$-stopping time
\[
 \tau_0 \dfn \inf\{t \geq 0 :  \indic_{\{\Theta \neq 0\}} W_t = 1\}.
 \]
 Consider also the nonnegative $(\filt{G}, \Qu)$-supermartingale 
 \begin{align*}
 	Z &\dfn \Exp\left(\int_0^\cdot  \frac{-\Theta}{1- W_u} \ud W_u\right) \indic_{\dbraco{0 , \tau_0}} \\
	&= (1 - W)^\Theta \exp\left(\frac{\Theta - \Theta^2}{2} \int_0^\cdot  \frac{1}{(1-W_u )^2 } \ud u\right)  \indic_{\dbraco{0 , \tau_0}}.
\end{align*}
Since $\int_0^\tau \Theta^2 (1-W_u)^{-2} \d u = \infty$ we have $Z$
is continuous by Larsson and Ruf \cite[Theorem~4.2]{Larsson:Ruf:convergence}. This then yields
that $Z$ is a  $(\filt{G}, \Qu)$-local martingale. We  assume from now on that $\Qu[\Theta \in \{0\} \cup [1/2, \infty)] = 1$, as this is a necessary and sufficient condition for $Z$ to be a  $(\filt{G}, \Qu)$-martingale, by the arguments in Ruf \cite{Ruf_martingale}.

Set now $S \dfn \Exp(W)$  and $\filt{F} \dfn \filt{F}^W$ and define the $\filt{F}$-predictable time
\[
 \tau_0^W \dfn \inf\{t \geq 0 :  W_t = 1\}.
\]
Then we obtain $K = \Qu[\Theta \geq 1/2] \indic_{\dbracc{\tau_0^W, \infty}}$ and $M = 1$.  Hence, by Proposition~\ref{P:180101.1}, we have that ${}^\circ Y^\Pu$ is  an  $(\filt{F}, \Pu)$-local martingale if and only if either $\Qu[\Theta \geq 1/2] = 1$ or $\Qu[\Theta = 0] = 1$. In the later case, ${}^\circ Y^\Pu=1$ is  an  $(\filt{F}, \Pu)$-martingale.

If $\mu[\ud \theta] \dfn \Qu[\Theta \in \ud \theta]$, $\theta \in \{0\} \cup [1/2, \infty)$,  describes the marginal law of $\Theta$, then 
\[
{}^\circ Z^\Qu = \Qu[\Theta = 0] + \int_{1/2}^{\infty}  (1 - W)^\theta \exp\left(\frac{\theta - \theta^2}{2} \int_0^\cdot  \frac{1}{(1-W_u )^2 } \ud u\right)  \indic_{\dbraco{0 , \tau_0}} \mu[\ud \theta]
\]
and hence
\begin{align*}
	{}^\circ Y^\Pu &= \left(\Qu[\Theta = 0] + \int_{1/2}^{\infty}  (1 - W)^\theta \exp\left(\frac{\theta - \theta^2}{2} \int_0^\cdot  \frac{1}{(1-W_u )^2 } \ud u\right)   \mu[\ud \theta]\right)^{-1}  \\
		&\times \indic_{\dbraco{0 , \tau_0^W}}  +  \indic_{\dbraco{\tau_0^W, \infty}}. 
\end{align*}
On the event $\{\tau_0^W < \infty\}$ we have ${}^\circ Y^\Pu_{\tau_0^W-} = 1/\Qu[\Theta = 0]$, $\Pu$-a.e., illustrating that if $\Qu[\Theta = 0]  > 0$,  then indeed ${}^\circ Y^\Pu$ is not an $(\filt{F}, \Pu)$-local martingale.
\end{example}

\begin{remark} 
	Under any of the conditions in Proposition~\ref{P:180101.1}, it holds that $M = 1$.  A general characterization of when exactly $M = 1$ holds eludes us at the time of writing. However, when $M = 1$ then $\tau_0$ is an $(\filt{F}, \Qu)$-pseudo-stopping time, meaning $\E_\Qu[N_\rho] = \E_\Qu[N_0]$ for each $(\filt{F}, \Qu)$-uniformly integrable martingale $N$; vice versa, if each $(\filt{F}, \Qu)$-martingale is continuous and $\tau_0$ is an $(\filt{F}, \Qu)$-pseudo stopping time, then $M = 1$. These facts follows from Nikeghbali and Yor \cite[Theorem~1]{Nikeghbali:Yor:05}. (The proof of \cite[Theorem~1]{Nikeghbali:Yor:05} only requires the continuity  of $(\filt{F}, \Qu)$-martingales in one direction; moreover, the assumption $\tau_0 < \infty$ in that paper can also omitted by a change-of-time argument.) In light of this fact, the previous section adds new examples of pseudo-stopping times to the literature.
\end{remark}

Thanks to Theorem~\ref{thm:1}, the process $M / {}^o Z^{\Qu}$ is of special interest, as it serves as an $\filt{F}$-local martingale deflator.

\begin{proposition} \label{P:180101.2}
	In the notation of this section, and under Assumption~\ref{Ass00}, the following statements concerning the   $(\filt{F}, \Pu)$-local martingale $M / {}^o Z^{\Qu}$  are equivalent.
	
\begin{enumerate}
	\item\label{P:180101.i} $M / {}^o Z^{\Qu}$ is an $(\filt{F}, \Pu)$-martingale.
	\item\label{P:180101.ii} $M$ is an $(\filt{F}, \Qu)$-martingale, and $\{{}^o Z^{\Qu} = 0 \} \subseteq \{M = 0\}$,  $\Qu$-a.s.
	\item\label{P:180101.iii} $M$ is an $(\filt{F}, \Qu)$-martingale, and $\{K = 1 \} \subseteq \{M = 0\}$, $\Qu$-a.s.
\end{enumerate}
\end{proposition}
\begin{proof}
	Note that
	\begin{align*}
		\E_\Pu\left[\frac{M_t }{ {}^o Z^{\Qu}_t}\right] &= \E_\Qu\left[M_t \indic_{\{{}^o Z^{\Qu}_t > 0\}}\right] = \E_\Qu[M_t] - \E_\Qu\left[M_t \indic_{\{ {}^o Z^{\Qu}_t = 0\}}\right], \quad t \geq 0.
	\end{align*}
	This yields the equivalence of statements \eqref{P:180101.i} and \eqref{P:180101.ii}. For the equivalence of statements \eqref{P:180101.ii} and \eqref{P:180101.iii}, one only needs to observe $\{{}^o Z^{\Qu} = 0\} = \{K = 1\} \cup \{M = 0\}$, thanks to \eqref{eq:180101.1} and \eqref{eq:180101.2}.  
	\end{proof}

We continue with a couple of examples. The first one involves a non-constant $M$ appearing as information on $\tau_0$ gets revealed in $\filt{F}$, and illustrates the necessary and sufficient conditions of Proposition~\ref{P:180101.2}.

\begin{example}    \label{Ex:180925}
Assume that the underlying probability space, equipped with the probability measure $\Qu$, supports a $\Qu$-Brownian motion $W$ and an independent  random variable $\Theta$ with $\Qu[\Theta = -1] = q/2 = \Qu[\Theta = 1]$ and $\Qu[\Theta = 0] = 1 - q$, where $q \in (0,1]$. Consider the filtration $\filt{G}$ to be the smallest right-continuous one that makes $W$ adapted, and such that $\Theta$ is $\mathcal{G}_0$-measurable. Moreover, consider the $\filt{G}$-stopping times
\begin{align*}
	\nu &\dfn  \inf\{t \geq 0 :  |W_t| = 1\},  \\
	\rho &\dfn \inf\{t \geq \nu :  W_t = - \sign (W_\nu) \}, \\
	\tau_0 &\dfn \inf\{t \geq 0 :  \Theta W_t = 1\}.	
\end{align*}
Note that $\Qu[0 < \nu < \rho < \infty] = 1$, and that $\tau_0$ equals $\Qu$-a.e.~either $\nu$, $\rho$ or $\infty$ (the latter infinite value if, and only if, $\Theta = 0$). Consider also the nonnegative $(\filt{G}, \Qu)$-martingale $Z \dfn 1 - \Theta W^{\tau_0}$.

Let $\filt{F}$ be  the smallest right-continuous filtration that makes $W$ adapted, and such that $\Theta$ is $\mathcal{F}_\rho$-measurable. Under $\filt{F}$, $\Theta$ is only revealed at $\rho$, as opposed to $\filt{G}$ where $\Theta$ is known from the beginning of time. We set $S \dfn \Exp(W)$, which is  both a $(\filt{G}, \Qu)$- and an $(\filt{F}, \Qu)$-martingale. We also note that $\nu$ and $\rho$ are $\filt{F}$-predictable  times.  Observe that, despite the different filtration structure, this example resembles Example~\ref{Ex:190823}. In both cases, $W$ is Brownian motion conditioned to never hit level $1 / \Theta$.

With the above set-up, we compute
\[
\Qu\left[\left. \tau_0 \leq t \right|  \sigalg{F}_t\right]  =  \frac{q}{2} \indic_{\dbraco{\nu , \rho}}(t) + \indic_{\{\Theta \neq 0\}} \indic_{\dbraco{\rho, \infty}}(t), \qquad t \geq 0.	
\]
Hence,  we have  
\[
K =\frac{q}{2}  \indic_{\dbraco{\nu , \rho}} + q \indic_{\dbraco{\rho, \infty}}; \qquad M = \indic_{\dbraco{0, \rho }}  + \frac{1}{1-q} \indic_{\{\Theta = 0\}} \indic_{\dbraco{\rho, \infty}},	
\]
with the understanding that $M = 1$ if $q = 1$. Note that $M$ is a  bounded  $(\filt{F}, \Qu)$-martingale. Moreover, straightforward computations give
\begin{align*}
	{}^o Z^{\Qu} = \indic_{\dbraco{0, \nu}} + \left(1 - q + \frac{q}{2} (1 + \sign (W_\nu) W) \right) \indic_{\dbraco{\nu, \rho }}  +  \indic_{\{\Theta = 0 \}}   \indic_{\dbraco{\rho , \infty}}.
\end{align*}	
Hence, when $q \in (0,1)$,
\begin{align*}
\frac{M}{{}^o Z^{\Qu}} = \indic_{\dbraco{0, \nu}} + \frac{1}{1 - q + (q / 2) (1 + \sign (W_\nu) W)} \indic_{\dbraco{\nu, \rho }}  +  \frac{1}{1-q} \indic_{\{\Theta = 0 \}}   \indic_{\dbraco{\rho , \infty}},
\end{align*}
$\Pu$-a.e., which is a  bounded  $(\filt{F}, \Pu)$-martingale; however, when $q = 1$, then
\[
\frac{M}{{}^o Z^{\Qu}} = \indic_{\dbraco{0, \nu}} + \frac{2}{1 + \sign (W_\nu) W}  \indic_{\dbraco{\nu, \infty }},	
\]
$\Pu$-a.e., 
which can be seen to be a strict local $(\filt{F}, \Pu)$-martingale. These observations are
 consistent with the result of Proposition~\ref{P:180101.2}.
\end{example}

We next modify Example~\ref{Ex:180925} to illustrate that it is also possible that the local martingale part   $M$ in the multiplicative decomposition of $(\Qu[\tau > t|\sigalg{F}_t])_{t \geq 0}$ is continuous. 
\begin{example}
	Assume that the underlying probability space, equipped with the probability measure $\Qu$, supports a pair of independent $\Qu$-Brownian motions $(W, B)$. Let $\filt{F} \dfn \filt{F}^{(W,B)}$ and let  $\filt{G}$ denote the smallest right-continuous filtration that makes $W$ adapted, and contains all the information of $B$ already at time $0$. Consider the process $\psi \dfn \sqrt{2} \int_0^\cdot \exp(-u) \ud B_u$, and note that $\Theta \dfn \psi_\infty$ is $\mathcal{G}_0$-measurable with standard normal distribution, and that the conditional law of $\Theta$ given $\mathcal{F}_t$ is Gaussian with mean $\psi_t$ and standard deviation $\exp(-t)$ for each $t \geq 0$. Set, as before, $\tau_0 \dfn \inf\{t \geq 0 :  \Theta W_t = 1  \}$, and consider the nonnegative $(\filt{G}, \Qu)$-martingale $Z \dfn 1 - \Theta W^{\tau_0}$.
	
	With $\underline{\theta} \dfn 1 / \inf_{u \in [0, \cdot]}  W_u$ and $\overline{\theta} \dfn 1 / \sup_{u \in [0, \cdot]}  W_u$, note that $\{\tau_0 > t\} = \{ \underline{\theta}_t < \Theta < \overline{\theta}_t\}$. It follows that		
	\[
	A_t \dfn \Qu\left[\left. \tau_0 > t \right|  \sigalg{F}_t\right]  = \Phi \left( \exp(t)(\overline{\theta}_t - \psi_t) \right) - \Phi \left( \exp(t)(\underline{\theta}_t - \psi_t) \right), \qquad t \geq 0.,
	\]
	where $\Phi$ denotes the standard normal distribution function.
	Writing the dynamics of the above, we see that the local martingale part in the additive decomposition of $A$ has non-zero quadratic variation everywhere. The same properties carry over to the multiplicative decomposition, yielding that $M$ is a Brownian local martingale with strictly increasing quadratic variation.
\end{example}

The next example has similar features as the setup of Sect.~\ref{sec:Bayes}, in the sense that he projection of the local martingale deflator loses mass whenever in the small filtration one learns about the sign of an excursion of a Brownian motion. This example also relates to the framework of the following section.

\begin{example}
Assume that the underlying probability space, equipped with the probability measure $\Qu$, supports a $\Qu$-Brownian motion $W$ and let $B \dfn \int_0^\cdot \sign(W_u) \ud W_u = |W| - \Lambda$ denote its L\'evy transformation, as mentioned in the introduction. Consider the filtrations $\filt{G} \dfn \filt{F}^W$ and $\filt{F} \dfn \filt{F}^B = \filt{F}^{|W|}$. Let us write
	\[
		\tau_0 \dfn \inf\{t \geq 0 :  1 + W_t - B_t = 0\} =   \inf\left\{t \geq 0 :  W_t = - \frac{1+\Lambda_t}{2}\right\},
	\]
	where $\Lambda$ denotes local time of $W$ at zero.
We now set $S \dfn \Exp(B)$ and consider the process
\[
Z \dfn \Exp\left (\int_0^\cdot  \frac{1}{1 + W_u - B_u} \ud B_u\right)  \indic_{\dbraco{0, \tau_0 }}.
\]	
We claim that $Z$ has continuous paths and is a $(\filt{G}, \Qu)$-martingale. To see path-continuity, note that just before $\tau_0$ the process $1 + W - B = 1 + W - |W|+\Lambda$ behaves like twice a Brownian motion hitting level zero, given that  $\Lambda$ will be flat (since $W$ is away from zero); then, it suffices to note that $\int_0^\cdot (1 + \beta_u)^{-2} \ud u$ explodes at the first time that a Brownian motion $\beta$ hits $-1$. Path-continuity of $Z$, coupled with its definition, implies that it is a $(\filt{G}, \Qu)$-local martingale. To see the actual martingale property of $Z$, we follow the arguments in  Ruf \cite{Ruf_martingale} as follows. Note that a continuous nonnegative local martingale $Z$ is also a local martingale in its own filtration (since one may choose the localising sequence to consist of level-crossing times); therefore, for proving that it is an actual martingale, which is equivalent to showing that it has constant expectation in time, one may assume that $Z$ lives on an appropriate canonical path-space, where results from F\"ollmer \cite{F1972} on change of measure can be utilised.
Consider then the F\"ollmer measure $\overline{\Pu}$, given by the extension of the measures defined via the Radon-Nikodym derivatives $(Z_{\tau_n \wedge n})_{n \in \Natural}$ on the increasing sequence  $(\sigalg F_{\tau_n \wedge n})_{n \in \Natural}$, where $(\tau_n)_{n \in \Natural}$ is a $\filt{G}$-localisation sequence for $Z$.   For some $\overline{\Pu}$-Brownian motion $U$ we then have
	\[
		W = \int_0^\cdot \frac{\sign(W_u)}{1 + W_u - B_u} \ud u + U.
	\]
	Hence, whenever $1+W-B$ becomes small then $W$ moves like a two-dimensional  $(\filt{G}, \overline{\Pu})$-Bessel process. In particular, $1-W-B$ never hits zero and $\int_0^\cdot  (1 + W_u - B_u)^{-2} \ud u < \infty$, $\overline{\Pu}$-a.e., yielding that $Z$ is indeed a martingale.

Let us now consider the $\filt{F}$-predictable  times $(\rho_i)_{i \in \mathbb N_0}$ and $(\tau_i)_{i \in \mathbb N}$, defined inductively by $\rho_0 \dfn 0$ and 
\[
	\tau_i \dfn \inf \left\{t >\rho_{i-1}:  |W_t| = \frac{1 + \Lambda_t}{2}\right\}; \quad \rho_i \dfn \inf \{t >\tau_i:  |W_t| = 0\}, \quad i \in \mathbb N.
\]
Then we have
\begin{align*}
	 \Qu\left[\left.\tau_0> t \right|  \sigalg{F}_t\right] = \left(\frac{1}{2}\right)^{\# \{i \in \mathbb{N}: \tau_i \leq t\}}, \qquad t \geq 0.
\end{align*}
Hence, by  \eqref{eq:180101.2}, we get
\begin{align*}
	 K_t = 1 -  \left(\frac{1}{2}\right)^{\# \{i \in \mathbb{N}: \tau_i \leq t\}}, \qquad t \geq 0.
\end{align*}
For the $\filt{F}$-optional $\Pu$-projection ${}^\circ Y^{\Pu}$ of the $\filt{G}$-local martingale deflator $Y = 1/Z$, we then have
$
	{}^\circ Y^{\Pu} = (1-K)/{{}^o Z^{\Qu}} ,
$
where $1/{}^o Z^{\Qu}$ is an $(\filt{F}, \Pu)$-martingale by Proposition~\ref{P:180101.2}.
\end{example}

In Example~\ref{Ex:180214.1}, it was shown that projections of reciprocals of $\filt{G}$-num\'eraires are not necessarily reciprocals of $\filt{F}$-num\'eraires. However, we have the following result.

\begin{proposition} \label{P:180107.1}
	In the notation of this section, suppose that  Assumption~\ref{Ass00} holds. Moreover, assume Jacod's hypothesis~(H) holds under $\Qu$; i.e., each $(\filt{F}, \Qu)$-martingale is also a $(\filt{G}, \Qu)$-martingale.   Then the following statements hold:
	\begin{enumerate}
		\item\label{P:180107.i}
			 $K_\rho \indic_{\{\rho < \infty\}} = {\Qu}[\tau_0 \leq \rho | \sigalg F_\infty] \indic_{\{\rho < \infty\}} $ for all $\filt{F}$-predictable times $\rho$  and $M = 1$.
		\item\label{P:180107.ii}
			 If $Y = 1/Z$ is a $\filt{G}$-num\'eraire deflator then $1/{}^o Z^{\Qu}$ is an $\filt{F}$-num\'eraire deflator. 
	\end{enumerate}
\end{proposition}

\begin{proof}
	For statement \eqref{P:180107.i}, it suffices to argue that ${\Qu}[\tau_0 \leq \cdot | \sigalg F_\infty]$ is the $\filt{F}$-predictable projection of the process $\indic_{\{\tau_0 \leq \cdot\}}$. That is, for a given $\filt{F}$-predictable time $\rho$ we need to argue that
	\begin{align} \label{eq:190914}
		 {\Qu}[\tau_0 \leq \rho | \sigalg F_\infty] \indic_{\{\rho < \infty\}} = {\Qu}[\tau_0 \leq \rho | \sigalg F_{\rho-}] \indic_{\{\rho < \infty\}}.
	\end{align}
	We shall argue this assertion by showing that the right-hand side is indeed the $\sigalg{F}_\infty$-conditional expectation of  $\indic_{\{\tau_0 \leq \rho < \infty\}}$. To this end, fix $A \in \sigalg{F}_\infty$ and note that
$
		\Qu[A | \sigalg{F}_{\rho-}]   = \Qu[A | \sigalg{G}_{\rho-}] 
$
	since each $(\filt{F}, \Qu)$-martingale is also a $(\filt{G}, \Qu)$-martingale by  assumption.	This now yields
	\begin{align*}
		\E_\Qu\left[{\Qu}[\tau_0 \leq \rho | \sigalg F_{\rho-}] \indic_{\{\rho < \infty\}} \indic_{A}\right]
			&= \E_\Qu \left[ {\Qu}[A | \sigalg F_{\rho-}] \indic_{\{\tau_0 \leq \rho < \infty\}}\right]
			\\
			&= \E_\Qu \left[ {\Qu}[A | \sigalg G_{\rho-}] \indic_{\{\tau_0 \leq \rho < \infty\}}\right] \\
			&=  \Qu \left[A \cap \{ \tau_0 \leq \rho < \infty\}\right],
	\end{align*}
	where the last equality uses the fact that $\tau_0$  is $\Qu$-a.e.~equal to a $\filt{G}$-predictable time since $Z$ does not jump to zero; see  Larsson and Ruf \cite[Lemma~3.5]{Larsson:Ruf:2017}. This now yields \eqref{eq:190914}.
	
	For statement \eqref{P:180107.ii}, observe that we have ${}^o Z^{\Qu} = \Exp(\int_0^\cdot \theta_u \ud  S_u) N$ for some nonnegative $\filt{F}$-predictable process $\theta$ and some  $(\filt{F}, \Qu)$-local martingale $N$ with $[N, S] = 0$. By assumption, $N$ is also a $(\filt{G}, \Qu)$-local martingale. By the product rule, so is $N Z$. Hence, $N$ is a nonnegative $(\filt{G}, \Pu)$-local martingale, thus an $(\filt{F}, \Pu)$-local martingale, as it is $\filt{F}$-adapted. Moreover, ${}^o Z^{\Qu} / N = \Exp(\int_0^\cdot \theta_u \ud  S_u)$ is an $(\filt{F}, \Qu)$-local martingale; hence $1/N$ is also an $(\filt{F}, \Pu)$-local martingale.  This implies that $N = 1$.
\end{proof}

\begin{remark}
	From a modelling point of view, it is convenient to observe that Jacod's hypothesis~(H) holds, for example, if 
	 $\filt{G}$ is of the form
	\begin{align*}
		\sigalg{G}_t \dfn \bigcap_{s > t} \left( \sigalg{F}_s \vee \sigalg{H}_s\right), \qquad t \geq 0,
	\end{align*}
	where $\filt{H}$ is a filtration such that $\sigalg{F}_\infty$ and $\sigalg{H}_\infty$ are independent under $\Qu$. Indeed, fix any $(\filt{F}, \Qu)$-martingale $N$, some $s,t \geq 0$ with $s < t$, and some $A \in \sigalg{H}_s$. Then
	\begin{align*}
		\E_\Qu[N_t  \indic_{A}] = \E_\Qu[N_t ] \Qu[A] =   \E_\Qu[N_s ] \Qu[A] = \E_\Qu[N_s \indic_{A}].
	\end{align*} 
	where we have used repeatedly the independence of  $\sigalg{F}_\infty$ and $\sigalg{H}_\infty$ under $\Qu$.  In particular, the Bayesian setup of Sect.~\ref{sec:Bayes} satisfies the assumptions of Proposition~\ref{P:180107.1}. 
	As a corollary, Assumption~\ref{Ass00} and Jacod's hypothesis~(H) holding under $\Qu$ do not imply that each $(\filt{F}, \Pu)$-martingale is also a $(\filt{G}, \Pu)$-martingale. For example, the process $1/\zeta$ in Sect.~\ref{sec:proof1} is an  $(\filt{F}, \Pu)$-local
martingale, but not a $(\filt{G}, \Pu)$-local martingale if $K^h_\infty > 0$.
\end{remark}

\section{Completeness and Such} \label{sec:counterexample}
\subsection{A motivating example}

We return to a question posed in the introduction: could a complete market become incomplete after shrinking the filtration.  We first provide a motivating example demonstrating that this is indeed possible. Theorem~\ref{thm:main} and Corollary~\ref{C:1} will also yield a slew of examples in a more comprehensive manner.

Let $W$ denote a standard Brownian motion and $B$ its L\'evy transformation, defined via
\[
B \dfn \int_0^\cdot \sign(W_u) \ud W_u.
\]
Consider the $\filt{F}^W$-stopping time 
\[
	\tau \dfn \inf \{t \geq 0:  W_t = 1\},
\]
noting that $\tau$ is not a stopping time under the filtration $\filt{F}^B = \filt{F}^{|W|}$.  Set ${\filt{G}} \dfn \filt{F}^W$ and $\filt{F} \dfn  \filt{F}^{B, \indic_{ \dbraco{\tau, \infty}}}$, the smallest right-continuous filtration that makes $B$ adapted and $\tau$ a stopping time. It follows that $\filt{F} \subseteq{\filt{G}}$, and that the (one-dimensional) stock price $S \dfn \Exp(B)$ is $\filt{F}$-adapted.  
Both $B$ and $S$ have the predictable representation property under $\filt{G}$, rendering market completeness under $\filt{G}$.  

Consider the $\filt{F}^B$-stopping times $(\rho_i)_{i \in \mathbb N_0}$ and $(\tau_i)_{i \in \mathbb N}$, defined inductively by $\rho_0 = 0$ and 
\[
	\tau_i \dfn \inf \{t >\rho_{i-1}:  |W_t| = 1\}; \qquad \rho_i \dfn \inf \{t >\tau_i:  |W_t| = 0\}, \qquad i \in \mathbb N.
\]
These stopping times allow to define the $\filt{F}$-adapted process
\begin{align*}
	N \dfn \indic_{ \dbraco{\tau, \infty}} - \frac{1}{2} \sum_{i \in \mathbb N} \indic_{ \dbraco{\tau_i, \infty}} \indic_{\{\tau_i \leq \tau\}},
\end{align*}
which is piecewise constant and jumps only at the times before $\tau$ when $|W|$ hits one. More precisely, $N$ jumps up or down by $1/2$ with probability $1/2$, depending on whether $W$ hits $1$ or $-1$; hence, it is an $\filt{F}$-martingale, but not a $\filt{G}$-local martingale. The discontinuous process $N$ \emph{a fortiori} cannot be expressed as a stochastic integral with respect to the geometric Brownian motion $S$; hence, the market is indeed incomplete under $\filt{F}$. 
Note that this observation is consistent with the martingale representation results in Br\'emaud and Yor \cite[Proposition~9]{Bremaud:Yor}.

\subsection{A more general construction} \label{SS:5.2}

The following result is of independent interest.

\begin{theorem} \label{thm:main}
	Let $W$ be a standard Brownian motion, and let $B$ denote its L\'evy transformation. Then, for \emph{any} given probability law $\mu$ on $((0, \infty], \B(0, \infty])$, there exists an $ \filt{F}^{W}$-stopping time $\tau$ with law $\mu$, independent of $\sigalg{F}^B_\infty$.
\end{theorem}

Theorem \ref{thm:main}, proved in \S \ref{SS:5.2bis} below, allows for an interesting class of examples where the completeness property fails through filtration shrinkage.

\begin{corollary} \label{C:1}
	There exist two nested filtrations $\filt{F} \subseteq\filt{G}$ and a one-dimensional continuous stock price process $S$, adapted to $\filt{F}$, such that the market is complete under $\filt{G}$ and under $\filt{F}^S$, but not under the ``intermediate information'' model $\filt{F}$.
\end{corollary}
\begin{proof}
	Using the above notation, set ${\filt{G}} \dfn \filt{F}^W$ and $S \dfn \Exp(B)$, where we recall that $\Exp(\cdot)$ denotes the stochastic exponential operator, and note that $\filt{F}^S = \filt{F}^B$.  Next, take $\mu$ to be the law of a non-deterministic distribution, and consider an $\filt{F}^W$-stopping time $\tau$ as in Theorem~\ref{thm:main}, with distribution $\mu$. Define now $\filt{F}$ to be the right-continuous modification of the progressive enlargement of the filtration $\filt{F}^B$ with the random time $\tau$. Clearly, $\filt{F}^S = \filt{F}^B \subseteq\filt{F} \subseteq\filt{F}^W = {\filt{G}}$, and $B$ is a Brownian motion on all three considered filtrations. However, although $B$ (hence $S$) has the predictable representation property on both $\filt{F}^B$ and $\filt{F}^W$, it loses the predictable representation property on $\filt{F}$. This can be readily seen by considering the (non-continuous) $\filt{F}$-local martingale $N$ defined via $N = \indic_{\dbraco{ \tau, \infty}} - C^\tau$, where $C^\tau$ denotes the compensator of $\indic_{\dbraco{ \tau, \infty}}$ under $\filt{F}$. 
\end{proof}

In the context of the proof of Corollary \ref{C:1}, the pair $(B, N)$ jointly has the predictable representation property on $\filt{F}$. It follows that every local martingale deflators in $\mathcal Y^{\filt{F}}$ is of the form $\indic_{\dbraco{0, \tau}}  + g(\tau) \indic_{\dbraco{\tau, \infty}}$ for some strictly positive Borel-measurable function $g: (0,\infty) \rightarrow (0,\infty)$ with $\int_0^\infty g(s) \mu [\ud s] = 1$.

\subsection{Proof of Theorem \ref{thm:main}} \label{SS:5.2bis}

Before we state the main auxiliary result in order to prove Theorem \ref{thm:main}, we discuss some prerequisites on excursions of Brownian motion. We keep as much as possible notation from Revuz and Yor \cite[Chapter XII]{RY}. For a continuous function $w: [0,\infty) \rightarrow \R$ with $w(0) = 0$, set $R(w) \dfn \inf \set{t > 0 \such w(t) = 0}$. Then, let $\mathcal{U}$ be the subset of all continuous functions $w$ such that $w(0) = 0$, $R(w) > 0$, and $w(t) = 0$ for all $t \geq R(w)$. We denote by $\mathcal{U}_+$ (respectively, $\mathcal{U}_-$) the subset of $\mathcal{U}$ with the extra property that $w(t) > 0$ (respectively, $w(t) < 0$) holds for all $t \in(0, R(w))$, in which case we speak of positive (respectively, negative) excursions. With $\delta : [0,\infty) \rightarrow \Real$ denoting the function that is identically equal to zero (which in particular implies that $\delta \notin \mathcal{U}$), consider the state space $\mathcal{U}_\delta \dfn \mathcal{U} \cup \set{\delta}$.

Recalling that $\Lambda$ is the local time of the Brownian motion $W$ at zero, define 
\[
\sigma_s \dfn \inf \set{t > 0 \such \Lambda_t > s}, \quad s \in [0, \infty),
\]
and note that this is a stopping time in $\filt{F}^{B}  = \filt{F}^{|W|}$ for all $s \in [0, \infty)$. We denote by $(e_s)_{s \in [0, \infty)}$ the excursion Poisson  point process of $W$. More precisely, for $s \in [0, \infty)$ with $\sigma_{s-} = \sigma_s$ we set $e_s \equiv \delta$, while if $\sigma_{s-} < \sigma_s$ then $e_s \in \mathcal{U}$ will be the excursion of $W$ over the interval $[\sigma_{s-}, \sigma_s]$, defined via $e_s (t) = W(\sigma_{s-} + t)$ for $t \in [0, \sigma_s - \sigma_{s-}]$, and $e_s (t) = 0$ for $t > \sigma_s - \sigma_{s-}$. 
We shall use $|e| \dfn (|e|_s)_{s \in [0, \infty)}$ to denote the process such that $|e|_s (t) = |e_s (t)|$ holds for all $s, t \geq 0$, and note that $|e|$ is also a Poisson point process, with state space $\mathcal{U}_+ \cup \set{\delta}$; in effect, $|e|$ forgets the excursion signs.

With the above notation set, we need the following facts:
\begin{enumerate}
	\item The sigma-algebra generated by the Poisson point process $|e| = (|e|_s)_{s \in [0, \infty)}$ coincides with $\sigalg{F}^{|W|}_\infty=\sigalg{F}^{B}_\infty$.
	\item Conditionally on the process $|e|$ (i.e., conditional on $\sigalg{F}^{|W|}_\infty$), the signs of the excursions are (a countable number of) independent and identically distributed random variables taking the values $-1$ and $+1$ with probability $1/2$.
\end{enumerate}

The first statement above is a consequence of \cite[Chapter XII, Proposition 2.5]{RY}. Indeed, since $\Lambda$ is $\sigalg{F}^{|W|}_\infty$-measurable, meaning $\Lambda_s$ and
$\sigma_s$ are $\mathcal F_\infty^{|W|}$-measurable for each $s \geq 0$, the process $(\sigma_s)_{s \geq 0}$ is also $\sigalg{F}^{|W|}_\infty$-measurable, and then it is straightforward that $(|e|_s)_{s \geq 0}$ is $\sigalg{F}^{|W|}_\infty$-measurable. On the other hand, one may reconstruct $|W|$ from $|e|$ as follows: first, for $s \geq 0$ one defines $\sigma_s \dfn \sum_{v \in (0, s]} R(|e|_v)$, then one obtains $\Lambda$ as the right-continuous inverse of $\sigma$, and then one defines $|W_t| \dfn |e|_{\Lambda_t} (t - \sigma(\Lambda_t -))$ for all $t \geq 0$.

The second statement above comes, for example, as a consequence of the discussion in Blumenthal \cite[Chapter IV, mostly page~114]{Blumenthal:1992}; see also Prokaj \cite{Prokaj:2009}.

The following result is the main tool in establishing the validity of Theorem \ref{thm:main}.

\begin{lemma} \label{lem:aux}
Fix a strictly decreasing sequence $(s_n)_{n \in \mathbb N}$ in $(0, \infty)$ such that $\lim_{n \uparrow \infty} s_n = 0$. Then, there exists a countable collection $(U_n)_{n \in \mathbb N}$ of random variables such that:
\begin{itemize}
	\item for each $\nin$, $U_n$ is $\sigalg{F}^W_{\sigma_{s_n}}$-measurable; and
	\item $(U_n)_{\nin}$ consists of independent and identically distributed random variables with the standard uniform law, and is furthermore independent from $\sigalg{F}^{|W|}_\infty=\sigalg{F}^{B}_\infty$.
\end{itemize}
\end{lemma}

\begin{proof}
For each $\nin$, write $\sigma_n \dfn \sigma_{s_n}$ for typographical simplicity. Consider the intervals $I_n \dfn (s_{n+1}, s_n]$; since the sequence $(s_n)_{\nin}$ is strictly decreasing, $(I_n)_{\nin}$ consists of disjoint intervals.

Define the random, $\Pu$-a.e.~countable set 
\[
	D \dfn \set{s \in [0, \infty) \such e_s \neq \delta} = \set{s \in [0, \infty) \such \sigma_{s-} < \sigma_s}\]
	where excursions actually happen in the local time clock.  The set $D \cap I_n$ corresponds to excursion times that happen in $I_n$ in the local time clock, and is clearly countable and infinite. Furthermore, $(e_s)_{s \in I_n}$ is $\sigalg{F}^W_{\sigma_n}$-measurable, for all $n \in \Natural$. It is straightforward (for example, by ordering the excursion sizes) to see that one may find an $\sigalg{F}^W_{\sigma_n}$-measurable enumeration $(v_{n, k})_{\kin}$ of $D \cap I_n$, for each $n \in \Natural$. Then, for all $\nin$, the $\sigalg{F}^W_{\sigma_n}$-measurable random variables $X_{n, k} \dfn \indic_{\{e_{v_{n, k}} \in \mathcal{U}_+\}}$ are $\set{0, 1}$-valued. Moreover, using the fact that the intervals $(I_n)_{\nin}$ are disjoint, we obtain that, conditional on $\sigalg{F}^{B}_\infty = \sigalg{F}^{|W|}_\infty$, the doubly-indexed collection $(X_{n, k})_{(n, k) \in \Natural \times \Natural}$ consists of independent and identically distributed random variables with 
	\[
		\Pu \bra{X_{n,k} = 0} = \frac{1}{2} = \Pu \bra{X_{n,k} = 1}.
	\] Therefore, upon defining $U_n \dfn \sum_{k=1}^{\infty}2^{-k} X_{n, k}$ for all $\nin$, we obtain a sequence $(U_n)_{\nin}$ of independent and identically distributed random variables with the standard uniform law, that are further independent of $\sigalg{F}^{|W|}_\infty$. Finally, since  $(X_{n, k})_{k \in \Natural}$ are $\sigalg{F}^W_{\sigma_n}$-measurable for each $\nin$, we obtain that $U_n$ is $\sigalg{F}^W_{\sigma_n}$-measurable for each $\nin$, which completes the argument.
\end{proof}

Given Lemma~\ref{lem:aux} above, we may now proceed to prove Theorem~\ref{thm:main}. 
\begin{proof}[of Theorem~\ref{thm:main}.]
Let $\mu$ be any probability law on $((0, \infty], \B((0, \infty]) )$. For any $s \in [0, \infty)$, let $\mu_s$ denote the probability law on $((0, \infty], \B((0, \infty]) )$ that is ``$\mu$ conditioned to be greater than $s$;'' more formally, if $\mu((s, \infty]) = 0$ then set $\mu_s( A) \dfn \indic_{\infty \in A}$ for all $A \in \B((0, \infty])$; otherwise,  if $\mu((s, \infty]) > 0$, set
\[
\mu_s (A) \dfn \frac{\mu(A \cap (s, \infty])}{\mu((s, \infty])}, \quad A \in \B((0, \infty]).
\]
Note that $(\mu_s)_{s \geq 0}$ is increasing in first-order stochastic dominance and that, as $s \downarrow 0$, $\mu_s$ converges (actually, in total variation) to $\mu_0 = \mu$.

Pick $(s_n)_{\nin}$ to be any strictly decreasing sequence of positive numbers with $\lim_{n \uparrow \infty} s_n = 0$. In the notation of Lemma~\ref{lem:aux}, consider a corresponding sequence $(U_n)_{\nin}$. We shall construct inductively a \emph{nonincreasing} sequence $(\tau_n)_{\nin}$ of $\filt{F}^W$-stopping times each having conditional law with respect to $\sigalg{F}^B_\infty$ equal to $\mu_{\sigma_{s_n}}$, and being measurable with respect to $\sigalg{F}^B_\infty \vee \sigma(U_m)_{m \leq n}$. 

As in the  proof of Lemma~\ref{lem:aux}, for each $\nin$ write $\sigma_n \dfn \sigma_{s_n}$, and note that $\sigma_n$ is a stopping time in $\filt{F}^B = \filt{F}^{|W|}$ for all $\nin$. Let $F_s : (0, \infty] \rightarrow [0,1]$ be the cumulative distribution function of $\mu_s$ defined via $F_s(t) \dfn \mu_s ((0, t])$ for all $t > 0$ and $s \geq 0$, and set $\tau_1 \dfn F_{\sigma_1}^{-1} (U_1)$, where we use the generalized right-continuous inverse. Since $\tau_1 > \sigma_1$ and $U_1$ is $\sigalg{F}^W_{\sigma_1}$-measurable, it is clear that $\tau_1$ is an $\filt{F}^W$-stopping time. By construction, and since $U_1$ is independent of $\sigalg{F}^{B}_\infty$, the conditional law of $\tau_1$ with respect to $\sigalg{F}^B_\infty$ equals $\mu_{\sigma_1}$. Fix now $n \in \Natural$, and suppose that one has constructed $\tau_n$, which is an $\filt{F}^W$-stopping time, measurable with respect to $\sigalg{F}^B_\infty \vee \sigma(U_m)_{m \leq n}$, and having conditional law with respect to $\sigalg{F}^B_\infty$ equal to $\mu_{\sigma_n}$. Set $\tau'_{n+1} \dfn F_{\sigma_{n+1}}^{-1} (U_{n+1})$ which, since $\tau'_{n+1} > \sigma_{n+1}$ and $U_{n+1}$ is $\sigalg{F}^W_{\sigma_{n+1}}$-measurable, is an $\filt{F}^W$-stopping time. Given $\sigalg{F}^B_\infty$, the conditional law of $\tau'_{n+1}$ is $\mu_{\sigma_{n+1}}$. Set then 
\[
	\tau_{n+1} \dfn \tau'_{n+1} \indic_{\tau'_{n+1} \leq \sigma_n} + \tau_{n} \indic_{\tau'_{n+1} > \sigma_n}. 
\]
Since $U_{n+1}$ is $\sigalg{F}^W_{\sigma_{n+1}}$-measurable and $\tau_n > \sigma_n$, it follows that $\tau_{n+1}$ is an $\filt{F}^W$-stopping time, and measurable with respect to $\sigalg{F}^B_\infty \vee (U_m)_{m \leq n + 1}$. Clearly, $\tau_{n+1} \leq \tau_n$. 
Furthermore, using conditional independence of $\tau'_{n+1}$ and $\tau_n$ given $\sigalg{F}^B_\infty$, which follows from the conditional independence of $U_{n+1}$ from $(U_m)_{m \leq n}$ given $\sigalg{F}^B_\infty$, and using the shorthand notation $\Pu^B \bra{\cdot} \dfn \Pu \bra{\cdot \such \sigalg{F}^B_\infty}$ for conditional probabilities given $\sigalg{F}^B_\infty = \sigalg{F}^{|W|}_\infty$, one obtains
\begin{align*} 
\Pu^B \bra{\tau_{n+1} \in A} &= \Pu^B \bra{\tau'_{n+1} \in A \cap (0, \sigma_n]} + \Pu^B \bra{ \tau'_{n+1} > \sigma_n, \tau_{n} \in A \cap (\sigma_{n}, \infty] } \\
&= \Pu^B \bra{\tau'_{n+1} \in A \cap (0, \sigma_n]} + \Pu^B \bra{\tau'_{n+1} > \sigma_n} \Pu^B \bra{\tau_{n} \in A \cap (\sigma_{n}, \infty]} \\
&= \mu_{\sigma_{n+1}} \pare{A \cap (0, \sigma_n]} + \mu_{\sigma_{n+1}} \pare{(\sigma_n, \infty]} \mu_{\sigma_n} \pare{A \cap (\sigma_n, \infty]} \\
&= \frac{\mu \pare{A \cap (\sigma_{n+1}, \sigma_n]}}{\mu((\sigma_{n+1}, \infty])} + \frac{\mu \pare{(\sigma_{n}, \infty]}}{\mu((\sigma_{n+1}, \infty])} \frac{\mu \pare{A \cap (\sigma_n, \infty]}}{\mu((\sigma_{n}, \infty])} \\
&= \frac{\mu \pare{A \cap (\sigma_{n+1}, \infty]}}{\mu((\sigma_{n+1}, \infty])} = \mu_{\sigma_{n+1}} (A), \quad A \in \B((0, \infty]),
\end{align*}
which implies that the conditional law of $\tau_{n+1}$ given $\sigalg{F}^B_\infty$ is $\mu_{\sigma_{n+1}}$. The inductive step is complete.

Define now $\tau \dfn \lim_{n \uparrow \infty} \tau_n = \bigwedge_{\nin} \tau_n$; since $\tau_n$ is an $\filt{F}^W$-stopping time for all $\nin$, $\tau$ is also an $\filt{F}^W$-stopping time.
Given that the conditional law of $\tau_{n}$ given $\sigalg{F}^B_\infty$ is $\mu_{\sigma_{n}}$ for each $\nin$ and that $(\sigma_n)_{\nin}$ decreases $\Pu$-a.e.~to zero, it follows that the conditional law of $\tau$ given $\sigalg{F}^B_\infty$ is $\mu_0 = \mu$. This implies both that $\tau$ is independent of $\sigalg{F}^B_\infty$, and that its probability law equals $\mu$, which concludes the proof of Theorem~\ref{thm:main}.
\end{proof}

\subsection{A further example where incompleteness arises though filtration shrinkage} \label{SS:5.3}

The markets described in Corollary~\ref{C:1} have an interesting ``quasi completeness'' property: for each $T \geq 0$, any nonnegative bounded $\sigalg{F}^S_T$-measurable contingent claim $\xi$ can by replicated; i.e., in the notation of Sect.~\ref{S:notation}
and with 
$x = x^{\filt{F}}(T, \xi)$ there exists a maximal 
$X \in\mathcal X^{\filt{F}}(x)$ 
such that $\Pu[X_T = \xi] = 1$.
See Ruf \cite[Section~3.4]{Ruf_ots} for a discussion of this weaker notion of completeness.  Below, we shall provide an example where the market is complete under the large filtration $\filt{G}$, but not even quasi-complete under the smaller filtration $\filt{F}$; this example answers negatively a conjecture by Jacod and Protter put forth in Jacod and Protter \cite{Jacod:Protter:2017}.

Let $B$ be a one-dimensional Brownian motion, and set $\filt{G}$ the right-continuous filtration generated by $B$; i.e., $\filt{G} \dfn \filt{F}^B$. 
Set $S \dfn \Exp(\int_0^\cdot \theta_u \ud B_u)$, where
\[
\theta_t \dfn \begin{cases}
B_t, \quad &t < 1; \\
1, \quad &t \geq  1, \ B_1 > 0; \\
2, \quad &t \geq  1, \ B_1 \leq 0. 
\end{cases}
\]
Since $\theta$ is nonzero, $(\Pu \times [B,B])$-a.e., the market is complete in $\filt{G}$.
 Now, let $\filt{F} \dfn \filt{F}^S$ be the right-continuous filtration generated by $S$. Since 
 \[
 	S_t = \exp\left( \frac{B_t^2 }{2} - \frac{1}{2} \int_0^t (1 + 
B_s^2)\right) \ud s
\]
 holds for all $t \leq 1$, it follows that $\sigalg{F}_t = \sigalg{F}^{|B|}_t$ holds for all $t < 1$. On the other hand, for all $t > 1$, since 
 \[
 	\set{B_1 > 0} = \set{[\log S,\log S]_t - [\log S,\log S]_1 = t - 1} \in \sigalg{F}_t, 
\]
it follows that $B_1 \in \sigalg{F}_t$. Furthermore, $B_s - B_1$ is also $\sigalg{F}_t$-measurable, which then implies that $(B_s)_{s \in [1, t]}$ is  $\sigalg{F}_t$-measurable, for all $t > 1$ and $s \in [1,t]$. Using right-continuity of the filtration, it then easily follows that, for all $t \geq 1$,  $\sigalg{F}_t$ is generated by $\sigalg{F}^{|B|}_1$ and $(B_s)_{s \in [1, t]}$.

Here, there is a jump of information at the filtration $\filt{F}$ happening exactly at the (deterministic) time $1$: the sign of $B$ at $t=1$ is suddenly revealed, while the only previous information was on the absolute value of $B$. In particular, any process of the form
\[
Z^\alpha \dfn 1 + \pare{\indic_{\{B_1 > 0\}} -\indic_{\{B_1 \leq 0\}}} \alpha  \indic_{\dbraco{1, \infty}}
\]
for all $\alpha \in (-1,1)$ is a strictly positive $\filt{F}$-martingale, and it is clearly purely discontinuous. Therefore, $S$ cannot have the predictable representation property in $\filt{F}$. 

To expand further, consider any claim of the form $f(\indic_{\{B_1 > 0\}})$, with delivery at time $1$, where $f: \set{0,1} \rightarrow [0, \infty)$. Its hedging cost  is $(f(0) + f(1)) / 2$ under the larger filtration $\filt{G}$. Under the filtration $\filt{F}$, its hedging cost is at least
\[
\sup_{\alpha \in (-1,1)} \E \bra{Z^\alpha_1 f(\indic_{\{B_1 > 0\}})} = \max \set{f(0), f(1)};
\]
in fact, this is the actual hedging cost since one may trivially hedge starting from this amount. This can also be argued in a ``dual'' way, observing that the probabilities $\Qu^\alpha$ constructed from $Z^\alpha$ for all $\alpha \in (-1,1)$ are exactly the class of equivalent local martingale measures under the filtration $\filt{F}$.

\begin{acknowledgements}
We would like to thank Umut Cetin, Claudio Fontana, Monique Jeanblanc, and Walter Schachermayer for discussions on the subject matter of this paper, as well as two anonymous referees and the Editor Martin Schweizer for their helpful comments.
\end{acknowledgements}

\bibliographystyle{spmpsci}      
\bibliography{aa_bib}   

\begin{thebibliography}{10}
\providecommand{\url}[1]{{#1}}
\providecommand{\urlprefix}{URL }
\expandafter\ifx\csname urlstyle\endcsname\relax
  \providecommand{\doi}[1]{DOI~\discretionary{}{}{}#1}\else
  \providecommand{\doi}{DOI~\discretionary{}{}{}\begingroup
  \urlstyle{rm}\Url}\fi

\bibitem{Acciaio:Fontana:Kardaras}
Acciaio, B., Fontana, C., Kardaras, C.: Arbitrage of the first kind and
  filtration enlargements in semimartingale financial models.
\newblock Stochastic Process. Appl. \textbf{126}(6), 1761--1784 (2016).
\newblock \doi{10.1016/j.spa.2015.12.004}.
\newblock \urlprefix\url{https://doi.org/10.1016/j.spa.2015.12.004}

\bibitem{Aksamit:2017}
Aksamit, A., Choulli, T., Deng, J., Jeanblanc, M.: No-arbitrage up to random
  horizon for quasi-left-continuous models.
\newblock Finance Stoch. \textbf{21}(4), 1103--1139 (2017).
\newblock \doi{10.1007/s00780-017-0337-3}.
\newblock \urlprefix\url{https://doi.org/10.1007/s00780-017-0337-3}

\bibitem{Aksamit:2018}
Aksamit, A., Choulli, T., Deng, J., Jeanblanc, M.: No-arbitrage under a class
  of honest times.
\newblock Finance Stoch. \textbf{22}(1), 127--159 (2018).
\newblock \doi{10.1007/s00780-017-0345-3}.
\newblock \urlprefix\url{https://doi.org/10.1007/s00780-017-0345-3}

\bibitem{Biagini:Mazzon}
Biagini, F., Mazzon, A., Perkki\"o, A.P.: Optional projection under equivalent
  local martingale measures (2020).
\newblock Preprint, arXiv:2003.09940

\bibitem{Bielecki:Jakubowski}
Bielecki, T.R., Jakubowski, J., Jeanblanc, M., Nieweglowski, M.:
  Semimartingales and shrinkage of filtration (2019).
\newblock Preprint, arXiv:1803.03700

\bibitem{Blumenthal:1992}
Blumenthal, R.M.: Excursions of {M}arkov {P}rocesses.
\newblock Probability and its Applications. Birkh\"auser Boston, Inc., Boston,
  MA (1992).
\newblock \doi{10.1007/978-1-4684-9412-9}.
\newblock \urlprefix\url{http://dx.doi.org/10.1007/978-1-4684-9412-9}

\bibitem{Bremaud:Yor}
Br{\'e}maud, P., Yor, M.: Changes of filtrations and of probability measures.
\newblock Z. Wahrsch. Verw. Gebiete \textbf{45}(4), 269--295 (1978).
\newblock \doi{10.1007/BF00537538}.
\newblock \urlprefix\url{http://dx.doi.org/10.1007/BF00537538}

\bibitem{Chau:Cosso:Fontana}
Chau, H.N., Cosso, A., Fontana, C.: The value of informational arbitrage.
\newblock Finance Stoch. \textbf{24}(2), 277--307 (2020).
\newblock \doi{10.1007/s00780-020-00418-3}.
\newblock \urlprefix\url{https://doi.org/10.1007/s00780-020-00418-3}

\bibitem{Chau:Runggaldier:Tankov}
Chau, H.N., Runggaldier, W.J., Tankov, P.: Arbitrage and utility maximization
  in market models with an insider.
\newblock Math. Financ. Econ. \textbf{12}(4), 589--614 (2018).
\newblock \doi{10.1007/s11579-018-0217-4}.
\newblock \urlprefix\url{https://doi.org/10.1007/s11579-018-0217-4}

\bibitem{Choulli:Stricker:1996}
Choulli, T., Stricker, C.: Deux applications de la d\'ecomposition de
  {G}altchouk-{K}unita-{W}atanabe.
\newblock In: S\'eminaire de {P}robabilit\'es, {XXX}, \emph{Lecture Notes in
  Math.}, vol. 1626, pp. 12--23. Springer, Berlin (1996).
\newblock \urlprefix\url{https://doi.org/10.1007/BFb0094638}

\bibitem{Coculescu:2012}
Coculescu, D., Jeanblanc, M., Nikeghbali, A.: Default times, no-arbitrage
  conditions and changes of probability measures.
\newblock Finance Stoch. \textbf{16}(3), 513--535 (2012).
\newblock \doi{10.1007/s00780-011-0170-z}.
\newblock \urlprefix\url{https://doi.org/10.1007/s00780-011-0170-z}

\bibitem{Elworthy_Li_Yor_99}
Elworthy, K.D., Li, X.M., Yor, M.: The importance of strictly local
  martingales; applications to radial {O}rnstein-{U}hlenbeck processes.
\newblock Probability Theory and Related Fields \textbf{115}, 325--355 (1999)

\bibitem{F1972}
F{\"o}llmer, H.: The exit measure of a supermartingale.
\newblock Z. Wahrscheinlichkeitstheorie und Verw. Gebiete \textbf{21}, 154--166
  (1972)

\bibitem{Foellmer_Protter_2010}
F\"ollmer, H., Protter, P.: Local martingales and filtration shrinkage.
\newblock ESAIM: Probability and Statistics \textbf{15}, 25--38 (2011)

\bibitem{Fontana:2018}
Fontana, C.: The strong predictable representation property in initially
  enlarged filtrations under the density hypothesis.
\newblock Stochastic Process. Appl. \textbf{128}(3), 1007--1033 (2018).
\newblock \doi{10.1016/j.spa.2017.06.015}.
\newblock \urlprefix\url{https://doi.org/10.1016/j.spa.2017.06.015}

\bibitem{Fontana:Jeanblanc:Song}
Fontana, C., Jeanblanc, M., Song, S.: On arbitrages arising with honest times.
\newblock Finance Stoch. \textbf{18}(3), 515--543 (2014).
\newblock \doi{10.1007/s00780-014-0231-1}.
\newblock \urlprefix\url{https://doi.org/10.1007/s00780-014-0231-1}

\bibitem{Gombani:Jaschke:Runggaldier}
Gombani, A., Jaschke, S., Runggaldier, W.: Consistent price systems for
  subfiltrations.
\newblock ESAIM Probab. Stat. \textbf{11}, 35--39 (2007).
\newblock \doi{10.1051/ps:2007004}.
\newblock \urlprefix\url{https://doi.org/10.1051/ps:2007004}

\bibitem{Jacod:Protter:2017}
Jacod, J., Protter, P.: Options prices in incomplete markets.
\newblock In: Enlargement of filtrations, \emph{ESAIM Proc. Surveys}, vol.~56,
  pp. 72--87. EDP Sci., Les Ulis (2017).
\newblock \doi{10.1051/proc/201756072}.
\newblock \urlprefix\url{https://doi.org/10.1051/proc/201756072}

\bibitem{Jeanblanc:Song:2015}
Jeanblanc, M., Song, S.: Martingale representation property in progressively
  enlarged filtrations.
\newblock Stochastic Process. Appl. \textbf{125}(11), 4242--4271 (2015).
\newblock \doi{10.1016/j.spa.2015.06.007}.
\newblock \urlprefix\url{https://doi.org/10.1016/j.spa.2015.06.007}

\bibitem{JYC_2009}
Jeanblanc, M., Yor, M., Chesney, M.: Mathematical {M}ethods for {F}inancial
  {M}arkets.
\newblock Springer Finance. Springer-Verlag London, Ltd., London (2009).
\newblock \doi{10.1007/978-1-84628-737-4}.
\newblock \urlprefix\url{https://doi.org/10.1007/978-1-84628-737-4}

\bibitem{Jeulin:Yor:1979}
Jeulin, T., Yor, M.: In\'egalit\'e de {H}ardy, semimartingales, et faux-amis.
\newblock In: S\'eminaire de {P}robabilit\'es, {XIII} ({U}niv. {S}trasbourg,
  {S}trasbourg, 1977/78), \emph{Lecture Notes in Math.}, vol. 721, pp.
  332--359. Springer, Berlin (1979)

\bibitem{Kailath:1971}
Kailath, T.: The structure of {R}adon-{N}ikod\'ym derivatives with respect to
  {W}iener and related measures.
\newblock Ann. Math. Statist. \textbf{42}, 1054--1067 (1971).
\newblock \urlprefix\url{https://doi.org/10.1214/aoms/1177693332}

\bibitem{KK}
Karatzas, I., Kardaras, C.: The num\'eraire portfolio in semimartingale
  financial models.
\newblock Finance and Stochastics \textbf{11}(4), 447--493 (2007)

\bibitem{KS1}
Karatzas, I., Shreve, S.E.: Brownian Motion and Stochastic Calculus,
  \emph{Graduate Texts in Mathematics}, vol. 113, second edn.
\newblock Springer-Verlag, New York (1991).
\newblock \doi{10.1007/978-1-4612-0949-2}.
\newblock \urlprefix\url{http://dx.doi.org/10.1007/978-1-4612-0949-2}

\bibitem{Kardaras_finitely}
Kardaras, C.: Finitely additive probabilities and the fundamental theorem of
  asset pricing.
\newblock In: Contemporary Quantitative Finance. Essays in Honour of Eckhard
  Platen, pp. 19--34. Springer, Berlin (2010)

\bibitem{Kardaras:Ruf:2018:filtration}
Kardaras, C., Ruf, J.: Projections of scaled {B}essel processes (2019).
\newblock \doi{10.1214/19-ECP246}.
\newblock \urlprefix\url{https://doi.org/10.1214/19-ECP246}

\bibitem{Larsson_2013}
Larsson, M.: Filtration shrinkage, strict local martingales and the {F}\"ollmer
  measure.
\newblock Annals of Applied Probability \textbf{24}(4), 1739--1766 (2014)

\bibitem{Larsson:Ruf:2017}
Larsson, M., Ruf, J.: Stochastic exponentials and logarithms on stochastic
  intervals---a survey.
\newblock J. Math. Anal. Appl. \textbf{476}(1), 2--12 (2019).
\newblock \doi{10.1016/j.jmaa.2018.11.040}.
\newblock \urlprefix\url{https://doi.org/10.1016/j.jmaa.2018.11.040}

\bibitem{Larsson:Ruf:convergence}
Larsson, M., Ruf, J.: Convergence of local supermartingales.
\newblock Annales de l'Institut Henri Poincar{\'e} (B) Probabilit{\'e}s et
  Statistiques \textbf{forthcoming} (2020)

\bibitem{Meyer:1973}
Meyer, P.A.: Sur un probl\`eme de filtration pp. 223--247. Lecture Notes in
  Math., Vol. 321 (1973)

\bibitem{Nikeghbali:essay}
Nikeghbali, A.: An essay on the general theory of stochastic processes.
\newblock Probab. Surv. \textbf{3}, 345--412 (2006).
\newblock \urlprefix\url{https://doi.org/10.1214/154957806000000104}

\bibitem{Nikeghbali:Yor:05}
Nikeghbali, A., Yor, M.: A definition and some characteristic properties of
  pseudo-stopping times.
\newblock Ann. Probab. \textbf{33}(5), 1804--1824 (2005).
\newblock \doi{10.1214/009117905000000297}.
\newblock \urlprefix\url{https://doi.org/10.1214/009117905000000297}

\bibitem{Perkowski_Ruf_2014}
Perkowski, N., Ruf, J.: Supermartingales as {R}adon-{N}ikodym densities and
  related measure extensions.
\newblock Ann. Probab. \textbf{43}(6), 3133--3176 (2015)

\bibitem{Prokaj:2009}
Prokaj, V.: Unfolding the {S}korohod reflection of a semimartingale.
\newblock Statist. Probab. Lett. \textbf{79}(4), 534--536 (2009).
\newblock \doi{10.1016/j.spl.2008.09.029}.
\newblock \urlprefix\url{http://dx.doi.org/10.1016/j.spl.2008.09.029}

\bibitem{Prokaj:Schachermayer:2011}
Prokaj, V., Schachermayer, W.: Hiding a constant drift---a strong solution.
\newblock Illinois J. Math. \textbf{54}(4), 1463--1480 (2012) (2010).
\newblock \urlprefix\url{http://projecteuclid.org/euclid.ijm/1348505537}

\bibitem{RY}
Revuz, D., Yor, M.: Continuous {M}artingales and {B}rownian {M}otion,
  \emph{Grundlehren der Mathematischen Wissenschaften [Fundamental Principles
  of Mathematical Sciences]}, vol. 293, third edn.
\newblock Springer-Verlag, Berlin (1999).
\newblock \doi{10.1007/978-3-662-06400-9}.
\newblock \urlprefix\url{http://dx.doi.org/10.1007/978-3-662-06400-9}

\bibitem{Ruf_ots}
Ruf, J.: Optimal {T}rading {S}trategies {U}nder {A}rbitrage.
\newblock Ph.D. thesis, Columbia University, New York, USA (2011).
\newblock Retrieved from http://academiccommons.columbia.edu/catalog/ac:131477

\bibitem{Ruf_martingale}
Ruf, J.: The martingale property in the context of stochastic differential
  equations.
\newblock Electronic Communications in Probability \textbf{20}(34), 1--10
  (2015)

\bibitem{Song:2016}
Song, S.: Drift operator in a viable expansion of information flow.
\newblock Stochastic Process. Appl. \textbf{126}(8), 2297--2322 (2016).
\newblock \doi{10.1016/j.spa.2016.02.001}.
\newblock \urlprefix\url{https://doi.org/10.1016/j.spa.2016.02.001}

\bibitem{Song:2017}
Song, S.: Martingale representation processes and applications in the market
  viability under information flow expansion.
\newblock In: Enlargement of {F}iltrations, \emph{ESAIM Proc. Surveys},
  vol.~56, pp. 111--135. EDP Sci., Les Ulis (2017)

\bibitem{Stricker:1977}
Stricker, C.: Quasimartingales, martingales locales, semimartingales et
  filtration naturelle.
\newblock Z. Wahrscheinlichkeitstheorie und Verw. Gebiete \textbf{39}(1),
  55--63 (1977).
\newblock \doi{10.1007/BF01844872}.
\newblock \urlprefix\url{https://doi.org/10.1007/BF01844872}

\bibitem{Stricker:Yan}
Stricker, C., Yan, J.A.: Some remarks on the optional decomposition theorem.
\newblock In: S\'eminaire de {P}robabilit\'es, {XXXII}, \emph{Lecture Notes in
  Math.}, vol. 1686, pp. 56--66. Springer, Berlin (1998).
\newblock \doi{10.1007/BFb0101750}.
\newblock \urlprefix\url{http://dx.doi.org/10.1007/BFb0101750}

\bibitem{Stricker:Yor}
Stricker, C., Yor, M.: Calcul stochastique d\'{e}pendant d'un param\`etre.
\newblock Z. Wahrsch. Verw. Gebiete \textbf{45}(2), 109--133 (1978).
\newblock \doi{10.1007/BF00715187}.
\newblock \urlprefix\url{https://doi.org/10.1007/BF00715187}

\end{thebibliography}

\end{document}